\documentclass[12pt,a4paper]{article}
\usepackage{amsmath}
\usepackage{amsthm}
\usepackage{amsfonts}
\usepackage{amssymb}
\usepackage{latexsym}

\DeclareFontFamily{U}{futm}{}
\DeclareFontShape{U}{futm}{m}{n}{  <-> s * [.92] fourier-bb}{}
\DeclareSymbolFont{Ufutm}{U}{futm}{m}{n}
\DeclareSymbolFontAlphabet{\mathbb}{Ufutm}

\newtheorem{theorem}{Theorem}
\newtheorem{lemma}[theorem]{Lemma}
\newtheorem{corollary}[theorem]{Corollary}

\newtheorem{remark}[theorem]{Remark}

\theoremstyle{definition}
\newtheorem{definition}[theorem]{Definition}

\begin{document}

\title{\Large{\textbf{Orthogonality of linear (alinear) quasigroups and their parastrophes}}}
\author{\normalsize {V.A.~Shcherbacov}}

 \maketitle

\begin{abstract}
\noindent
 Necessary and sufficient conditions of orthogonality of left (right) linear (alinear) quasigroups in
 various combinations are given.
As corollary we obtain conditions of parastroph orthogonality of left (right) linear (alinear) quasigroups.
Any linear (alinear) quasigroup over the  group $S_n$ ($n\neq 2; 6$) is not orthogonal to its $(12)$-parastrophe.

\medskip


\medskip

\end{abstract}

\tableofcontents

\bigskip

\section{Introduction}

In \cite{GONS_MARKOV_NECHAEV}  complete $k$-recursive MDS-codes are constructed using the systems of $n$-ary ($n\geq 2$) orthogonal quasigroups. Systems of orthogonal $n$-ary operations ($n\geq 2$) are used by construction of some crypto-algorithms  \cite{PS_VS_11, VS_CRYPTO_12}.
Therefore the study of quasigroup orthogonality is important from theoretical and "practical" point of view.

In introduction we give some basic definition. For more detailed information on basic concepts used in the paper it is possible to see \cite{VD, 1a, HOP, SCERB_03}.

\begin{definition} \label{def2_2} A binary  groupoid $(Q, A)$ with binary operation $A$ such
that in the equality $A(x_1, x_2) = x_3$ knowledge of any two  the  elements  $x_1, x_2, x_3$
 uniquely specifies the remaining one is called a binary quasigroup  \cite{2}.
\end{definition}

\begin{definition} \label{def2_2_2}
From Definition \ref{def2_2} it  follows that  with a given binary quasigroup $(Q, A)$   it possible to associate
$(3!-1)$ others, so-called parastrophes of quasigroup $(Q, A)$:
\begin{equation*}
\begin{split}
& A(x_1, x_2) = x_3 \Leftrightarrow \\
&  A^{(12)}(x_2, x_1) = x_3 \Leftrightarrow \\
& {A}^{(13)}(x_3, x_2) = x_1 \Leftrightarrow \\
& {A}^{(23)}(x_1, x_3) = x_2 \Leftrightarrow \\
& {A}^{(123)}(x_2, x_3) = x_1 \Leftrightarrow \\
& {A}^{(132)}(x_3, x_1) = x_2.
\end{split}
\end{equation*}
\cite[p. 230]{STEIN}, \cite[p. 18]{VD}.
\end{definition}
Notice, cases 5 and 6 are "$(12)$-parastrophes" of cases 3 and 4, respectively.

\begin{remark}
Sometimes the following definition of parastrophy is more convenient: all the same as in Definition \ref{def2_2_2}  for exception of the last two cases ${A}^{(123)}(x_3, x_1) = x_2 \Leftrightarrow {A}^{(132)}(x_2, x_3) = x_1$, i.e., cases 5 and 6 are "$(12)$-parastrophes" of cases 4 and 3, respectively.
\end{remark}

Here we follow tradition.

\begin{definition}
Let $(Q, +)$ be a quasigroup. A permutation $\overline{\varphi}$ of the set $Q$ is called an anti-automorphism of quasigroup $(Q, +)$, if the following equality is true for all $x, y \in Q$: $\overline{\varphi}(x+y) = \overline{\varphi}y + \overline{\varphi}x$.
\end{definition}

 Denote by $Aaut(Q, +)$ the set of all anti-automorphisms of  a quasigroup $(Q, +)$.

\begin{definition}\label{LEFT_LIN_MIDDLE_ELEM}
A quasigroup $(Q,\cdot)$ is called left linear, if $x\cdot y = \varphi x + a + \beta
y$, where $(Q,+)$ is a group, the element $a$ is a fixed element of the set $Q$, $\varphi \in Aut(Q,+)$, $\beta \in S_Q$.
\index{quasigroup!left linear over a group}
\end{definition}

\begin{lemma} \label{OTHER_FORM_OF_LEFT_LIN_QUAS}
If a left linear quasigroup $(Q, \cdot)$ has the form $x\cdot y = \varphi x + a + \beta y$ over a group $(Q,+)$, then it also has the form $x\cdot y = \varphi x +  J_a\beta y + a$ over the  group $(Q,+)$, where $J_a x = a + x - a$, and vice versa.
\end{lemma}
\begin{proof}
Indeed, from equality $x\cdot y = \varphi x +  a + \beta y$ we have $x\cdot y = \varphi x +  a + \beta y - a + a = \varphi x +  (a + \beta y - a) + a = \varphi x +  J_a\beta y + a$. Notice, the map $J_a$ is an inner automorphism of the group $(Q, \cdot)$ \cite{KM}. It is clear that the map $J_a$ is a permutation of the set $Q$.
\end{proof}

\begin{definition}\label{ALIN_MIDDLE_ELEM}
A quasigroup $(Q,\cdot)$ is called left alinear, if $x\cdot y = \overline{\varphi} x + a + \beta y$, where
$(Q,+)$ is a group, the element $a$ is a fixed element of the set $Q$, $\overline{\varphi}$ is an anti-automorphism, $\beta \in S_Q$.
\index{quasigroup!left alinear over a group}
\end{definition}

\begin{lemma} \label{OTHER_FORM_OF_LEFT_ALIN_QUAS}
If a left alinear quasigroup $(Q, \cdot)$ has the form $x\cdot y = \overline{\varphi} x + a + \beta y$ over a group $(Q,+)$, then it also has the form $x\cdot y = \overline{\varphi} x +  J_a\beta y + a$ over the  group $(Q,+)$, where $J_a x = a + x - a$, and vice versa.
\end{lemma}
\begin{proof}
The proof is similar to the proof of Lemma \ref{OTHER_FORM_OF_LEFT_LIN_QUAS}.
\end{proof}

In \cite{SM_00} left (right) linear quasigroups over an abelian group are called \textit{semicentral}.
\index{quasigroup!semicentral}

\begin{definition} \label{LIN_MIDDLE_ELEM}
\begin{enumerate}
\item
A quasigroup $(Q, \cdot)$ of the form $x\cdot y = \varphi x + a + \psi y$, where $(Q, +)$ is a group, the element $a$ is a fixed element of the set $Q$, $\varphi, \psi \in Aut (Q,+)$, is called  linear quasigroup (over the group $(Q, +)$).
\item
A quasigroup $(Q, \cdot)$ of the form $x\cdot y = \varphi x + a + \overline{\psi} y$, where $(Q, +)$ is a group, the element $a$ is a fixed element of the set $Q$, $\varphi \in Aut (Q,+), \overline{\psi} \in Aaut(Q,+)$, is called  left linear right alinear quasigroup (over the group $(Q, +)$).
\item
A quasigroup $(Q, \cdot)$ of the form $x\cdot y = \overline{\varphi} x + a + \psi y$, where $(Q, +)$ is a group, the element $a$ is a fixed element of the set $Q$, $\psi \in Aut (Q,+), \overline{\varphi} \in Aaut(Q,+)$, is called left alinear right linear quasigroup (over the group $(Q, +)$).
\item
A quasigroup $(Q, \cdot)$ of the form $x\cdot y = \overline{\varphi} x + a + \overline{\psi} y$, where $(Q, +)$ is a group, the element $a$ is a fixed element of the set $Q$, $\overline{\varphi}, \overline{\psi} \in Aaut (Q,+)$, is called alinear quasigroup (over the group $(Q, +)$).
\end{enumerate}
\index{quasigroup!linear over a group}
\end{definition}

\begin{definition}
A quasigroup $(Q, \cdot)$ of the form $x\cdot y = \varphi x +  \psi y + a$, where $(Q, +)$ is an abelian group, the element $a$ is a fixed element of the set $Q$, $\varphi, \psi \in Aut (Q,+)$, is called  T-quasigroup \cite{pntk, tkpn}.
\end{definition}

In Definition \ref{LIN_MIDDLE_ELEM} we follow tradition but below, using Lemma \ref{OTHER_FORM_OF_LEFT_LIN_QUAS}, we pass to other form of linear (alinear) quasigroups more usable in "abelian" case, i.e.,  more usable by the study of $T$-quasigroups.

\begin{remark}
From Lemma \ref{OTHER_FORM_OF_LEFT_LIN_QUAS} it follows:
 \begin{enumerate}
\item
 a linear quasigroup $(Q, \cdot)$ has the form $x\cdot y = \varphi x + a + \psi y$ if and only if it has the following form: $x\cdot y = \varphi x + I_a \psi y + a$.
 \item
a left linear right alinear quasigroup $(Q, \cdot)$ has  the form $x\cdot y = \varphi x + a + \overline{\psi} y$ if and only if it has the form $x\cdot y = \varphi x +  I_a \overline{\psi} y + a$.
\item
a left alinear right linear quasigroup $(Q, \cdot)$ has  the form $x\cdot y = \overline{\varphi} x + a + \psi y$ if and only if it has the form
$x\cdot y = \overline{\varphi} x +  I_a\psi y + a$
\item
an  alinear quasigroup  $(Q, \cdot)$ has the form $x\cdot y = \overline{\varphi} x + a + \overline{\psi} y$ if and only if it has the form $x\cdot y = \overline{\varphi} x + I_a\overline{\psi} y + a$.
  \end{enumerate}
\end{remark}

We give some elementary properties of quasigroup (group) automorphism and anti-automorphisms.

\begin{lemma} \label{ANTIAUTOM_PROP}
\begin{enumerate}
\item
The product of two anti-automorphisms of a quasigroup $(Q, +)$, say $\overline{\varphi}$ and  $\overline{\psi}$, is an automorphism of $(Q, +)$. \item
If $\varphi$ is an automorphism of a quasigroup $(Q, +)$ and $\overline{\psi}$ is its anti-automorphism, then $\varphi \overline{\psi}$, $\overline{\psi}\varphi$ are some  anti-automorphisms of quasigroup $(Q, +)$.
\item
Denote by the letter $I$ the following anti-automorphism of a group $(Q, +)$: $I(x) = -x$ for any $x \in Q$. It is well known that $I^2 = \varepsilon$ \cite{KM}.
Any anti-automorphism $\overline{\psi}$ of a group $(Q, +)$ can be presented in the form $\overline{\psi} = I \psi$, where $\psi \in Aut(Q, +)$.
\item $I\varphi = \varphi I$.
\item $I\overline{\varphi} = \overline{\varphi} I$.
\item $(\overline{\varphi})^{-1} = \overline{\varphi^{-1}}= I \varphi^{-1}$.
\item By $J_a$ we denote inner automorphism of a group $(Q, +)$, i.e. $J_a x = a + x - a$ for any $x\in Q$. If $\varphi \in Aut (Q, +)$, then $\varphi J_a = J_{\varphi a} \varphi$,  $\varphi J^{-1}_a = J^{-1}_{\varphi a} \varphi$.
\item If  $J_a \in Inn(Q, +)$ and  $I\varphi \in Aaut(Q, +)$, then $I\varphi J_a = J_{\varphi a}I\varphi$, i.e., $\overline{\varphi} J_a = J_{\varphi a}\overline{\varphi}$.
\item $IJ_a  =  J_{a} I$.
\end{enumerate}
\end{lemma}
\begin{proof}
\begin{enumerate}
\item
Indeed, $\overline{\varphi}\overline{\psi}(x + y) = \overline{\varphi}(\overline{\psi}y + \overline{\psi}x) = \overline{\varphi}\overline{\psi}x + \overline{\varphi} \overline{\psi}y$.
\item
This is easy to check.
\item
From the last equality we have  the following $I\overline{\psi} = \psi$, that proves this statement.
\item We have
$0 = \varphi I(x + Ix) = \varphi(x + Ix) = \varphi x + \varphi I x$. From other side $0 =  \varphi x + I \varphi x$. Therefore $\varphi I = I\varphi$.
\item
We have $ \overline{\varphi} I = I \varphi I = \varphi$, $I  \overline{\varphi}  = I^2 \varphi  = \varphi$. Therefore
  $\overline{\varphi} I = I\overline{\varphi}$.
\item $(\overline{\varphi})^{-1} = (I\varphi)^{-1} = \varphi^{-1} I = I \varphi^{-1}  = \overline{\varphi^{-1}}$.
\item We have $\varphi J_a x = \varphi (a + x - a) = \varphi a + \varphi x - \varphi a = J_{\varphi a} \varphi$.
\item We have $ \overline{\varphi}J_a x  = I\varphi J_a x = I \varphi (a + x - a) = I(\varphi a + \varphi x - \varphi a) =
- I\varphi a + I\varphi x  + I\varphi a = \varphi a + I\varphi x  -\varphi a = J_{\varphi a} I\varphi x = J_{\varphi a} \overline{\varphi} x$.
\item Indeed,  $ IJ_a x  =  I (a + x - a) = I(- a)  + I x + I a =  a + I x  -  a = J_{a} I x$.
\end{enumerate}
\end{proof}

Below we shall use properties described in Lemma \ref{ANTIAUTOM_PROP} without additional comments.

It is clear that any alinear quasigroup over an abelian group is linear since in any abelian group any antiautomorphism is an automorphism.

\begin{lemma} \label{FORMS_LEFT_LIN_QUAS}
\begin{enumerate}
\item For any left linear quasigroup  $(Q,\cdot)$ there exists its form such that  $x\cdot y = \varphi x + \beta y$.
\item For any right linear quasigroup  $(Q,\cdot)$ there exists its form such that  $x\cdot y = \alpha x + \psi y$.
\item For any left alinear quasigroup  $(Q,\cdot)$ there exists its form such that  $x\cdot y = \overline{\varphi} x + \beta y$.
\item For any right linear quasigroup  $(Q,\cdot)$ there exists its form such that  $x\cdot y = \alpha x + \overline{\psi} y$.
\end{enumerate}
\end{lemma}
\begin{proof}
\begin{enumerate}
\item
We can re-write the form $x\cdot y = \varphi x + \beta y+c$ of a left linear quasigroup  $(Q,\cdot)$ as follows
$x\cdot y = \varphi x + R_c\beta y = \varphi x + \beta^{\,\prime} y,$ where $\beta^{\,\prime} = R_c\beta$.

\item
We can re-write the form $x\cdot y = \alpha  x + \psi y+c$ of a right  linear quasigroup  $(Q,\cdot)$ as follows
$x\cdot y = \alpha  x + c -c + \psi y+c  = R_c\alpha  x + I_c\psi y = \alpha^{\,\prime} x + \psi^{\, \prime} y$,
where $I_{-c}\psi y = -c + \psi y +c$.
\item The proof is similar to the proof of Case 1.
\item The proof is similar to the proof of Case 2.
\end{enumerate}
\end{proof}

\begin{remark}
In general Lemma \ref{FORMS_LEFT_LIN_QUAS} is not true for linear, left linear right alinear, left alinear right linear, and alinear quasigroups.
\end{remark}

\section{Parastrophes of left(right) linear (alinear) quasigroups}

\begin{lemma} \label{Forms_of_parastrophes_OF_LIN_QUAS}
Suppose that quasigroup $(Q, \cdot)$ is linear with the form $x\cdot y = \varphi x + \psi y + c$ over a group $(Q, +)$.
Then its parastrophes have the following forms:
\begin{enumerate}
\item
$x\overset{(12)}{\cdot} y = \varphi y + \psi x + c$;
\item
$x\overset{(13)}{\cdot} y  = \varphi^{-1} x + I J_{I\varphi^{-1} c}\varphi^{-1} \psi y + I\varphi^{-1}c$, where  $J_{I\varphi^{-1} c} x =  -\varphi^{-1}c + x + \varphi^{-1} c$;
\item
$x\overset{(23)}{\cdot} y = I\psi^{-1}\varphi x +  \psi^{-1} y + I \psi^{-1}c$;
\item
$x\overset{(123)}{\cdot} y  = \varphi^{-1} y + I J_{I\varphi^{-1} c}\varphi^{-1} \psi x + I\varphi^{-1}c$;
\item
$x\overset{(132)}{\cdot} y = I\psi^{-1}\varphi y +  \psi^{-1} x + I \psi^{-1}c$.
\end{enumerate}
\end{lemma}
\begin{proof}
Case 1 is clear.

Case 2. By definition of parastrophy $x\cdot y = z \Leftrightarrow z\overset{(13)}{\cdot} y  = x$.
From equality $x\cdot y = \varphi x + \psi y + c = z$ we have $\varphi x = z-c -\psi y$, $ x =  z\overset{(13)}{\cdot} y = \varphi^{-1}z - \varphi^{-1} c - \varphi^{-1}\psi y$. If we replace  the letter $z$ by the letter $x$, we obtain $x \overset{(13)}{\cdot} y = \varphi^{-1}x - \varphi^{-1} c - \varphi^{-1}\psi y = \varphi^{-1}x - \varphi^{-1} c - \varphi^{-1}\psi y + \varphi^{-1} c - \varphi^{-1} c = \varphi^{-1}x + J_{I\varphi^{-1} c}I\varphi^{-1}\psi y + I\varphi^{-1} c =  \varphi^{-1}x + IJ_{I\varphi^{-1} c}\varphi^{-1}\psi y + I\varphi^{-1} c$.

Case 3. By definition of parastrophy $x\cdot y = z \Leftrightarrow x\overset{(23)}{\cdot} z  = y$.
From equality $x\cdot y = \varphi x + \psi y + c = z$ we have $\psi y = -\varphi x + z-c $, $ y =  x\overset{(23)}{\cdot} z =
 \psi^{-1}I\varphi x + \psi^{-1} z+ \psi^{-1}Ic$. If we replace  the letter $z$ by the letter $y$, we obtain $x\overset{(23)}{\cdot} y =
 \psi^{-1}I\varphi x + \psi^{-1} y + \psi^{-1}Ic =  I\psi^{-1}\varphi x + \psi^{-1} y + I\psi^{-1}c$.

Cases 4 and 5 are "$(12)$-parastrophes" of Cases 2 and 3, respectively.
\end{proof}

Recall, by Lemma  \ref{ANTIAUTOM_PROP},  $\overline{\varphi} = I\varphi$, $\overline{\psi} = I\psi$.  Below we shall use this relations without additional comments.

\begin{lemma} \label{Forms_of_parastrophes_OF_ALIN_QUAS}
Suppose that quasigroup $(Q, \cdot)$ is alinear with the form $x\cdot y = \overline{\varphi} x + \overline{\psi} y + c$ over a group $(Q, +)$.
Then its parastrophes have the following forms:
\begin{enumerate}
\item
$x\overset{(12)}{\cdot} y = \overline{\varphi} y + \overline{\psi} x + c$;
\item
$x\overset{(13)}{\cdot} y  = I\varphi^{-1} \psi y +  I J_{\varphi^{-1} c}\varphi^{-1} x + \varphi^{-1}c$;
\item
$x\overset{(23)}{\cdot} y = IJ_{\psi^{-1}c}\psi^{-1} y + IJ_{\psi^{-1}c} \psi^{-1} \varphi x + \psi^{-1}c$;
\item
$x\overset{(123)}{\cdot} y  = I\varphi^{-1} \psi x +  I J_{\varphi^{-1} c}\varphi^{-1} y + \varphi^{-1}c$;
\item
$x\overset{(132)}{\cdot} y = IJ_{\psi^{-1}c}\psi^{-1} x + IJ_{\psi^{-1}c} \psi^{-1} \varphi y + \psi^{-1}c$.
\end{enumerate}
\end{lemma}
\begin{proof}
Case 1 is clear.

Case 2. By definition of parastrophy $x\cdot y = z \Leftrightarrow z\overset{(13)}{\cdot} y  = x$.
From equality $x\cdot y = \overline{\varphi} x + \overline{\psi} y + c = z$ we have $\overline{\varphi} x = z - c - \overline{\psi} y$,   $ x =  z\overset{(13)}{\cdot} y =   - I\varphi^{-1}\overline{\psi} y - I{\varphi}^{-1} c + I{\varphi}^{-1}z = \varphi^{-1}\overline{\psi} y + {\varphi}^{-1} c - {\varphi}^{-1}z = I\varphi^{-1}{\psi} y + {\varphi}^{-1} c + I{\varphi}^{-1}z = I\varphi^{-1}{\psi} y + J_{{\varphi}^{-1} c}I{\varphi}^{-1}z + {\varphi}^{-1} c$.

If we replace  the letter $z$ by the letter $x$, we obtain $x \overset{(13)}{\cdot} y = I \varphi^{-1}{\psi} y + I J_{{\varphi}^{-1} c}{\varphi}^{-1}x + {\varphi}^{-1} c$.

Case 3. By definition of parastrophy $x\cdot y = z \Leftrightarrow x\overset{(23)}{\cdot} z  = y$.
From equality $x\cdot y = I\varphi x + I\psi y + c = z$ we have $I\psi y = -I\varphi x + z-c $,
$ y =  x\overset{(23)}{\cdot} z =
 I\psi^{-1} ( -I\varphi x + z-c ) = \psi^{-1}c + I\psi^{-1} z +I\psi^{-1} \varphi x$.

 If we replace  the letter $z$ by the letter $y$, we obtain $
 x\overset{(23)}{\cdot} y = \psi^{-1}c + I\psi^{-1} y +I\psi^{-1} \varphi x = J_{\psi^{-1}c}I\psi^{-1} y + J_{\psi^{-1}c} I\psi^{-1} \varphi x + \psi^{-1}c = IJ_{\psi^{-1}c}\psi^{-1} y + I J_{\psi^{-1}c} \psi^{-1} \varphi x + \psi^{-1}c$.

Cases 4 and 5 are "$(12)$-parastrophes" of Cases 2 and 3, respectively.
\end{proof}
\begin{remark}
From Lemma \ref{Forms_of_parastrophes_OF_ALIN_QUAS} it follows that any parastrophe of an alinear quasigroup is alinear.
\end{remark}

\begin{lemma} \label{Forms_of_parastrophes_OF_LIN_ALIN_QUAS}
Suppose that  $(Q, \cdot)$ is left linear right alinear quasigroup with the form $x\cdot y = \varphi x + I\psi y + c$ over a group $(Q, +)$.
Then its parastrophes have the following forms:
\begin{enumerate}
\item
$x\overset{(12)}{\cdot} y = \varphi y + I\psi x + c$;
\item
$x\overset{(13)}{\cdot} y  = \varphi^{-1} x +  J_{I{\varphi^{-1}} c}\varphi^{-1} \psi y + I\varphi^{-1}c$;
\item
$x\overset{(23)}{\cdot} y = IJ_{\psi^{-1}c} \psi^{-1} y + J_{\psi^{-1} c} \psi^{-1}\varphi x + \psi^{-1}c$;
\item
$x\overset{(123)}{\cdot} y  = \varphi^{-1} y +  J_{I{\varphi^{-1}} c}\varphi^{-1} \psi x + I\varphi^{-1}c$;
\item
$x\overset{(132)}{\cdot} y = IJ_{\psi^{-1}c} \psi^{-1} x + J_{\psi^{-1} c} \psi^{-1}\varphi y + \psi^{-1}c$.
\end{enumerate}
\end{lemma}
\begin{proof}
Case 1 is clear.

Case 2.
From equality $x\cdot y = \varphi x + \overline{\psi} y + c = z$ we have $\varphi x = z - c - \overline{\psi} y$,   $ x =  z\overset{(13)}{\cdot} y =   \varphi^{-1} z - {\varphi}^{-1} c + {\varphi}^{-1}\psi y = \varphi^{-1} z + J_{I\varphi^{-1} c}\varphi^{-1} \psi y + I\varphi^{-1} c$.

If we replace  the letter $z$ by the letter $x$, we obtain $x \overset{(13)}{\cdot} y = \varphi^{-1} x + J_{I\varphi^{-1} c}\varphi^{-1} \psi y + I\varphi^{-1} c$.

Case 3.
From equality $x\cdot y = \varphi x + I\psi y + c = z$ we have $I\psi y = I\varphi x + z + Ic $,
$ y =  x\overset{(23)}{\cdot} z =
 IJ_{\psi^{-1}c} \psi^{-1} z +J_{\psi^{-1} c} \psi^{-1}\varphi x + \psi^{-1}c$.

 If we replace  the letter $z$ by the letter $y$, then we obtain $
 x\overset{(23)}{\cdot} y =
  IJ_{\psi^{-1}c} \psi^{-1} y + J_{\psi^{-1} c} \psi^{-1}\varphi x + \psi^{-1}c$.

Cases 4 and 5 are "$(12)$-parastrophes" of Cases 2 and 3, respectively.
\end{proof}

\begin{lemma} \label{Forms_of_parastrophes_OF_ALIN_LIN_QUAS}
Suppose that  $(Q, \cdot)$ is left alinear right linear quasigroup with the form $x\cdot y = I\varphi x + \psi y + c$ over a group $(Q, +)$.
Then its parastrophes have the following forms:
\begin{enumerate}
\item
$x\overset{(12)}{\cdot} y = I\varphi y + \psi x + c$;
\item
$x\overset{(13)}{\cdot} y  =\varphi^{-1} \psi y +I J_{\varphi^{-1} c}\varphi^{-1} x  + \varphi^{-1} c$;
\item
$x\overset{(23)}{\cdot} y = \psi^{-1} \varphi x + \psi^{-1}y + I\psi^{-1}c$;
\item
$x\overset{(123)}{\cdot} y  = \varphi^{-1} \psi x + I J_{\varphi^{-1} c}\varphi^{-1} y  + \varphi^{-1} c$;
\item
$x\overset{(132)}{\cdot} y = \psi^{-1} \varphi y + \psi^{-1}x + I\psi^{-1}c$.
\end{enumerate}
\end{lemma}
\begin{proof}
Case 1 is clear.

Case 2.
We have   $ x =  z\overset{(13)}{\cdot} y =   \varphi^{-1} \psi y +I J_{\varphi^{-1} c}\varphi^{-1} z  + \varphi^{-1} c$.
If we replace  the letter $z$ by the letter $x$, we obtain $x \overset{(13)}{\cdot} y = \varphi^{-1} \psi y +I J_{\varphi^{-1} c}\varphi^{-1} x  + \varphi^{-1} c$.

Case 3.
From equality $x\cdot y = I\varphi x + \psi y + c = z$ we have $y = \psi^{-1} \varphi x + \psi^{-1}z + I\psi^{-1}c$.

 If we replace  the letter $z$ by the letter $y$, we obtain $
 x\overset{(23)}{\cdot} y = \psi^{-1} \varphi x + \psi^{-1}y + I\psi^{-1}c$.

Cases 4 and 5 are "$(12)$-parastrophes" of Cases 2 and 3, respectively.
\end{proof}

\section{Orthogonality of left (right)  linear (alinear) quasigroups} \label{ORTHO_T_QUAS}

In this section we give  conditions of orthogonality of a pair of left (right)  linear quasigroups over a group $(Q, +)$.

\begin{definition}
Binary groupoid  $(G, \circ)$ is isotopic image of a binary groupoid  $(G, \cdot)$, if there exist permutations $\alpha, \beta, \gamma$ of the set $Q$ such that $x\circ y = \gamma^{-1}(\alpha x \cdot \beta y)$. The ordered  triple of permutations $(\alpha, \beta, \gamma)$ of the set $Q$
is called an \textit{isotopy}.
\end{definition}

In Lemma  \ref{ISOM_ORTH_SQUARES} a square is the inner part of Cayley table of a finite groupoid.
\begin{lemma}\label{ISOM_ORTH_SQUARES}
Squares $S_1(Q_1)$ and $S_2(Q_2)$ are orthogonal if and only if their isotopic images are orthogonal with the
isotopies of the form $T_1=(\varepsilon, \varepsilon, \varphi)$ and $T_2 = (\varepsilon, \varepsilon, \psi)$,
respectively \cite[Lemma 7]{MS05}.
\end{lemma}

\begin{lemma} \label{UP_TO_THIRD_COPMONENT}
By the study of orthogonality of left (right)  linear (alinear) quasigroups of quasigroups $(Q, \cdot)$ and $(Q, \circ)$ of the forms $x \cdot y = \alpha x + \beta y + c$ and $x \circ y = \gamma x + \delta y + d$  we can take $c= d = 0$  without loss of generality.
\end{lemma}
\begin{proof}
 The inner part of Cayley table of any quasigroup is a square in the sense of Lemma \ref{ISOM_ORTH_SQUARES}.
Quasigroup $(Q, \cdot)$ is isotope of the form $(\varepsilon, \varepsilon, R^{-1}_c)$ of  quasigroup $(Q, \diamond)$
with the form $x\diamond y = \alpha x + \beta y$.

Quasigroup $(Q, \circ)$ is isotope of the form $(\varepsilon, \varepsilon, R^{-1}_d)$ of  quasigroup $(Q, \ast)$
with the form $x\ast y = \gamma x + \delta y$.

 Therefore  by the study of orthogonality of left (right)  linear (alinear) quasigroups we can take $ c = d = 0$ without loss of generality.
\end{proof}

By Lemma \ref{FORMS_LEFT_LIN_QUAS} any left linear quasigroup  $(Q,\cdot)$ over a group $(Q, +)$ has the form  $x\cdot y = \varphi x + \beta y$, where $\varphi \in Aut (Q,+)$, $\beta \in S_Q$.

\begin{theorem}\label{LEFT_LINEAR_ORTHOGON} Left linear quasigroups $(Q,\cdot)$ and $(Q,\circ)$ of the form $x\cdot y = \varphi  x + \beta y$ and
$x\circ y = \psi  x + \delta y$, respectively,  which are defined  over a
 group $(Q,+)$, are orthogonal if and only if the mapping   $( - \varphi^{-1}\beta   + \psi^{-1}\delta )$ is a permutation of the set $Q$. \end{theorem}
\begin{proof}
We follow \cite[Theorem 7]{SOKH_FRYZ_12}.
Quasigroups $(Q, \cdot)$ and  $(Q,\circ)$ are orthogonal if and only if the system of equations
\begin{equation} \label{LEFT_LIN_ORTH}
\left\{
\begin{split}
& \varphi x + \beta  y = a \\
& \psi  x + \delta  y  = b
\end{split}
\right.
\end{equation}
 has a unique solution for any fixed elements $a, b \in Q$.
We solve this system of equations in the usual way.
\[
\begin{array}{lll}
\left\{
\begin{array} {l}
x + \varphi^{-1}\beta  y   =  \varphi^{-1}a \\
x +  \psi^{-1}\delta  y     = \psi^{-1}b
\end{array}
\right. & \Longleftrightarrow & \left\{
\begin{array} {l}
 - \varphi^{-1}\beta  y - x  =  - \varphi^{-1}a \\
x +  \psi^{-1}\delta  y     = \psi^{-1}b
\end{array}\right.
\end{array}
\]
We perform  the  following transformation: (I row + II  row $\rightarrow$ I row) and obtain the system:
$$\left\{
\begin{array} {l}
 - \varphi^{-1}\beta  y + \psi^{-1}\delta  y = - \varphi^{-1}a + \psi^{-1}b\\
x +  \psi^{-1}\delta  y     = \psi^{-1}b
\end{array}
\right.
$$

Write expression $ - \varphi^{-1}\beta  y + \psi^{-1}\delta  y $ as follows $ (- \varphi^{-1}\beta   + \psi^{-1}\delta) y $.
Then the system (\ref{LEFT_LIN_ORTH}) is equivalent to the following system
 $$\left\{
\begin{array} {l}
( - \varphi^{-1}\beta   + \psi^{-1}\delta ) y = - \varphi^{-1}a + \psi^{-1}b\\
x +  \psi^{-1}\delta  y     = \psi^{-1}b
\end{array}
\right.
$$
It is clear that  the system (\ref{LEFT_LIN_ORTH}) has a unique solution if and only if the mapping  $( - \varphi^{-1}\beta   + \psi^{-1}\delta )$ is a permutation of the set $Q$.
\end{proof}

\begin{theorem}\label{LEFT_LINEAR(12)_ORTHOGON} Linear quasigroups $(Q,\cdot)$ and   left linear quasigroup $(Q,\circ)$ of the form $x\cdot y = \varphi  x + \beta y + c$ and
$x\circ y = \psi  y + \delta x$, respectively,  which are defined  over a
 group $(Q,+)$, are orthogonal if and only if the mapping  $(J_{\psi^{-1}b} \varphi^{-1}\delta   - \beta^{-1} \varphi)$ is a  permutation of the set $Q$ for any $b\in Q$.
  \end{theorem}
\begin{proof}
Quasigroups $(Q, \cdot)$ and  $(Q,\circ)$ are orthogonal if and only if the system of equations
\begin{equation} \label{LEFT_LIN_ORTH_1}
\left\{
\begin{split}
& \varphi x + \beta  y = a-c \\
& \psi  y + \delta  x  = b
\end{split}
\right.
\end{equation}
 has a unique solution for any fixed elements $a, b \in Q$.
We solve this system of equations in the usual way.
\[
\begin{array}{lll}
\left\{
\begin{array} {l}
I y + I \beta^{-1} \varphi x   =  I \beta^{-1}(a-c) \\
y +  \psi^{-1}\delta  x     = \psi^{-1}b
\end{array}
\right. & \Longleftrightarrow & \left\{
\begin{array} {l}
I y + I \beta^{-1} \varphi x   =  I \beta^{-1}(a-c) \\
J_y \psi^{-1}\delta  x + y     = \psi^{-1}b
\end{array}\right.
\end{array}
\]
We perform  the  following transformation: (II row + I  row $\rightarrow$ I row) and obtain the system:
\begin{equation} \label{LIN+ALIN}
\left\{
\begin{split}
& J_y \psi^{-1}\delta  x + I \beta^{-1} \varphi x = \psi^{-1}b + I \beta^{-1}(a-c) \\
& y     = \psi^{-1}b - J_y \psi^{-1}\delta  x
\end{split}
\right.
\end{equation}

Substitute the right side  of the second equation of the system (\ref{LIN+ALIN}) in the first equation of this system instead of the variable $y$:

Then the system (\ref{LIN+ALIN}) is equivalent to the following system
 $$\left\{
\begin{array} {l}
\psi^{-1}b -\varphi^{-1}\delta x + \varphi^{-1}\delta x + \varphi^{-1}\delta x -\psi^{-1} b - \beta^{-1} \varphi x  = \psi^{-1}b + I \beta^{-1}(a-c)\\
y     = \psi^{-1}b - J_y \psi^{-1}\delta  x
\end{array}
\right.
$$
Further we have
 $$\left\{
\begin{array} {l}
J_{\psi^{-1}b} \varphi^{-1}\delta x  - \beta^{-1} \varphi x  = \psi^{-1}b + I \beta^{-1}(a-c)\\
y     = \psi^{-1}b - J_y \psi^{-1}\delta  x
\end{array}
\right.
$$

Write expression $(J_{\psi^{-1}b} \varphi^{-1}\delta x  - \beta^{-1} \varphi x)$ as follows $(J_{\psi^{-1}b} \varphi^{-1}\delta  - \beta^{-1} \varphi )x$ and remember that $y +  \psi^{-1}\delta  x   = \psi^{-1}b$.
\begin{equation} \label{LIN+ALIN_1}
\left\{
\begin{split}
& (J_{\psi^{-1}b} \varphi^{-1}\delta   - \beta^{-1} \varphi) x  = \psi^{-1}b - \beta^{-1}(a-c)\\
& y     = \psi^{-1}b - \psi^{-1}\delta  x
\end{split}
\right.
\end{equation}

The systems (\ref{LEFT_LIN_ORTH_1}) and  (\ref{LIN+ALIN_1}) are equivalent.
It is clear that  the system (\ref{LEFT_LIN_ORTH_1}) has a unique solution if and only if the mapping  $(J_{\psi^{-1}b} \varphi^{-1}\delta   - \beta^{-1} \varphi)$ is a  permutation of the set $Q$ for any $b\in Q$.
\end{proof}

\begin{remark}
If in conditions of Theorem \ref{LEFT_LINEAR(12)_ORTHOGON} $(Q, +)$ is an abelian group, then expression $(J_{\psi^{-1}b} \varphi^{-1}\delta   - \beta^{-1} \varphi)$ takes the form $(\varphi^{-1}\delta   - \beta^{-1} \varphi)$, since in any abelian group any inner automorphism is the identity automorphism.
\end{remark}
\begin{remark}
Even in the case, when the group $(Q, +)$ is a finite cyclic  group, the solution of equation $\psi^{-1} \delta x - \beta^{-1} \varphi x = \psi^{-1}b- \beta^{-1} (a-c)$ is sufficiently complicate computational problem, since in general case  permutation $\delta$ is not an  automorphism of the group $(Q, +)$, This remark also applies to the similar theorems that are given below.
\end{remark}

By Lemma \ref{FORMS_LEFT_LIN_QUAS} any right linear quasigroup  $(Q,\cdot)$ over a group $(Q, +)$  has the form  $x\cdot y = \alpha  x + \varphi y$, where  $\alpha \in S_Q$,  $\varphi \in Aut (Q,+)$.

\begin{theorem}\label{Right_LINEAR_ORTHOGON} Right  linear quasigroups $(Q,\cdot)$ and $(Q,\circ)$ of the form $x\cdot y = \alpha  x + \varphi y$ and
$x\circ y = \gamma  x + \psi y$, respectively, which are defined over a
 group $(Q,+)$, are orthogonal if and only if the mapping   $(\varphi^{-1} \alpha   - \psi^{-1} \gamma)$ is a permutation of the set $Q$. \end{theorem}
 \begin{proof}
Quasigroups $(Q, \cdot)$ and  $(Q,\circ)$ are orthogonal if and only if the system of equations
\begin{equation} \label{Right_LIN_ORTH}
\left\{
\begin{split}
& \alpha x + \varphi y = a \\
& \gamma x + \psi y  = b
\end{split}
\right.
\end{equation}
 has a unique solution for any fixed elements $a, b \in Q$.
We solve this system of equations in the usual way.
\[
\begin{array}{lll}
\left\{
\begin{array} {l}
\varphi^{-1} \alpha x +  y   =  \varphi^{-1}a \\
\psi^{-1} \gamma x +   y     = \psi^{-1}b
\end{array}
\right. & \Longleftrightarrow & \left\{
\begin{array} {l}
\varphi^{-1} \alpha x +  y   =  \varphi^{-1}a \\
- y - \psi^{-1} \gamma x     =  - \psi^{-1}b
\end{array}\right.
\end{array}
\]
We do  the  following transformation: (I row + II  row $\rightarrow$ I row) and obtain the system:
$$\left\{
\begin{array} {l}
\varphi^{-1} \alpha x  - \psi^{-1} \gamma x    =  \varphi^{-1}a  - \psi^{-1}b \\
- y - \psi^{-1} \gamma x     =  - \psi^{-1}b.
\end{array}
\right.
$$

Write expression $\varphi^{-1} \alpha x  - \psi^{-1} \gamma x $ as follows $ (\varphi^{-1} \alpha   - \psi^{-1} \gamma ) x $.
Then the system (\ref{Right_LIN_ORTH}) is equivalent to the following system
 $$\left\{
\begin{array} {l}
(\varphi^{-1} \alpha   - \psi^{-1} \gamma) x    =  \varphi^{-1}a  - \psi^{-1}b \\
- y - \psi^{-1} \gamma x     =  - \psi^{-1}b.
\end{array}
\right.
$$
It is clear that  the system (\ref{Right_LIN_ORTH}) has a unique solution if and only if the mapping  $(\varphi^{-1} \alpha   - \psi^{-1} \gamma)$ is a permutation of the set $Q$.
\end{proof}

\begin{theorem}\label{Right_ALINEAR_ORTHOGON} Right  alinear quasigroups $(Q,\cdot)$ and $(Q,\circ)$ of the form $x\cdot y = \alpha  x + \overline{\varphi} y$ and
$x\circ y = \gamma  x + \overline{\psi} y$, respectively, which are defined over a
 group $(Q,+)$, are orthogonal if and only if the mapping   $(-(\overline{\varphi})^{-1}\alpha   + (\overline{\psi})^{-1}\gamma )$ is a permutation of the set $Q$. \end{theorem}
\begin{proof}
Quasigroups $(Q, \cdot)$ and  $(Q,\circ)$ are orthogonal if and only if the system of equations
\begin{equation} \label{Right_ALIN_ORTH}
\left\{
\begin{split}
& \alpha x + \overline{\varphi}  y = a \\
& \gamma  x + \overline{\psi}  y  = b
\end{split}
\right.
\end{equation}
 has a unique solution for any fixed elements $a, b \in Q$.
We solve this system of equations in the usual way.
\[
\begin{array}{lll}
\left\{
\begin{array} {l}
y + (\overline{\varphi})^{-1}\alpha  x   =  (\overline{\varphi})^{-1}a \\
y  +  (\overline{\psi})^{-1}\gamma  x     = (\overline{\psi})^{-1}b
\end{array}
\right. & \Longleftrightarrow & \left\{
\begin{array} {l}
I(\overline{\varphi})^{-1}\alpha  x + Iy  =  I(\overline{\varphi})^{-1}a \\
y  +  (\overline{\psi})^{-1}\gamma  x     = (\overline{\psi})^{-1}b.
\end{array}\right.
\end{array}
\]
We do  the  following transformation: (I row + II  row $\rightarrow$ I row) and obtain the system:
$$\left\{
\begin{array} {l}
I(\overline{\varphi})^{-1}\alpha  x + (\overline{\psi})^{-1}\gamma  x  = I(\overline{\varphi})^{-1}a + (\overline{\psi})^{-1}b\\
y  +  (\overline{\psi})^{-1}\gamma  x     = (\overline{\psi})^{-1}b.
\end{array}
\right.
$$

Write expression $I(\overline{\varphi})^{-1}\alpha  x + (\overline{\psi})^{-1}\gamma  x $ as follows $( -(\overline{\varphi})^{-1}\alpha   + (\overline{\psi})^{-1}\gamma ) x $.
Then the system (\ref{Right_ALIN_ORTH}) is equivalent to the following system
$$\left\{
\begin{array} {l}
(-(\overline{\varphi})^{-1}\alpha   + (\overline{\psi})^{-1}\gamma ) x = -(\overline{\varphi})^{-1}a + (\overline{\psi})^{-1}b\\
y  +  (\overline{\psi})^{-1}\gamma  x     = (\overline{\psi})^{-1}b.
\end{array}
\right.
$$
It is clear that  the system (\ref{Right_ALIN_ORTH}) has a unique solution if and only if the mapping  $(-(\overline{\varphi})^{-1}\alpha   + (\overline{\psi})^{-1}\gamma )$ is a permutation of the set $Q$.
\end{proof}

\begin{theorem}\label{LEFT_ALINEAR_ORTHOGON} Left alinear quasigroups $(Q,\cdot)$ and $(Q,\circ)$ of the form $x\cdot y = \overline{\varphi}  x + \beta y$ and
$x\circ y = \overline{\psi}  x + \delta y$, respectively,  which are defined  over a
 group $(Q,+)$, are orthogonal if and only if the mapping   $((\overline{\varphi})^{-1}\beta  - (\overline{\psi})^{-1}\delta )$ is a permutation of the set $Q$. \end{theorem}
\begin{proof}
Quasigroups $(Q, \cdot)$ and  $(Q,\circ)$ are orthogonal if and only if the system of equations
\begin{equation} \label{LEFT_ALIN_ORTH}
\left\{
\begin{split}
& \overline{\varphi} x + \beta  y = a \\
& \overline{\psi}  x + \delta  y  = b
\end{split}
\right.
\end{equation}
 has a unique solution for any fixed elements $a, b \in Q$.
We solve this system of equations in the usual way:
\[
\begin{array}{lll}
\left\{
\begin{array} {l}
(\overline{\varphi})^{-1}\beta  y  + x  =  (\overline{\varphi})^{-1}a \\
(\overline{\psi})^{-1}\delta  y  + x   = (\overline{\psi})^{-1}b
\end{array}
\right. & \Longleftrightarrow &
\left\{
\begin{array} {l}
(\overline{\varphi})^{-1}\beta  y  + x  =  (\overline{\varphi})^{-1}a \\
Ix + I(\overline{\psi})^{-1}\delta  y = I(\overline{\psi})^{-1}b
\end{array}
\right.
\end{array}
\]
We do  the  following transformation: (I row + II  row $\rightarrow$ II row) and obtain the system:
$$
\left\{
\begin{array} {l}
(\overline{\varphi})^{-1}\beta  y  + x  =  (\overline{\varphi})^{-1}a \\
(\overline{\varphi})^{-1}\beta  y + I(\overline{\psi})^{-1}\delta  y = (\overline{\varphi})^{-1}a - (\overline{\psi})^{-1}b
\end{array}
\right.
$$

Write expression $(\overline{\varphi})^{-1}\beta  y + I(\overline{\psi})^{-1}\delta  y $ as follows $((\overline{\varphi})^{-1}\beta - (\overline{\psi})^{-1}\delta) y $.
Then the system (\ref{LEFT_ALIN_ORTH}) is equivalent to the following system:
 $$
 \left\{
\begin{array} {l}
(\overline{\varphi})^{-1}\beta  y  + x  =  (\overline{\varphi})^{-1}a \\
((\overline{\varphi})^{-1}\beta  - (\overline{\psi})^{-1}\delta)  y = (\overline{\varphi})^{-1}a - (\overline{\psi})^{-1}b
\end{array}
\right.
$$
It is clear that  the system (\ref{LEFT_ALIN_ORTH}) has a unique solution if and only if the mapping  $((\overline{\varphi})^{-1}\beta  - (\overline{\psi})^{-1}\delta)$ is a permutation of the set $Q$.
\end{proof}

\begin{remark}
The mapping $((\overline{\varphi})^{-1}\beta  - (\overline{\psi})^{-1}\delta )$ from Theorem \ref{LEFT_ALINEAR_ORTHOGON} it is possible to write also in the form $((\overline{\varphi})^{-1}\beta  - (\overline{\psi})^{-1}\delta ) = I\varphi^{-1}\beta  - I\psi^{-1}\delta = I(-\psi^{-1}\delta + \varphi^{-1}\beta)$.

\end{remark}

\begin{theorem}\label{LEFT_LIN_Right_ALINEAR} Left linear quasigroup $(Q,\cdot)$ and right  alinear quasigroup $(Q,\circ)$ of the form $x\cdot y = \varphi  x + \beta y$ and $x\circ y = \gamma  y + \overline{\psi} x$, respectively, which are defined over a
 group $(Q,+)$, are orthogonal if and only if the mapping   $(\psi^{-1} \gamma + \varphi^{-1}\beta )$ is a permutation of the set $Q$.
 \end{theorem}
\begin{proof}
Quasigroups $(Q, \cdot)$ and  $(Q,\circ)$ are orthogonal if and only if the system of equations
\begin{equation} \label{Right_ALIN__LEFT_LIN ORTH}
\left\{
\begin{split}
& \varphi  x + \beta y = a \\
& \gamma  y + \overline{\psi} x  = b
\end{split}
\right.
\end{equation}
 has a unique solution for any fixed elements $a, b \in Q$.
We solve this system of equations in the usual way.
\[
\begin{array}{lll}
\left\{
\begin{array} {l}
x + {\varphi}^{-1}\beta  y   =  \varphi^{-1}a \\
x + (\overline{\psi})^{-1} \gamma  y    = (\overline{\psi})^{-1} b
\end{array}
\right. & \Longleftrightarrow & \left\{
\begin{array} {l}
I{\varphi}^{-1}\beta  y + Ix  =  I\varphi^{-1}a \\
x + (\overline{\psi})^{-1} \gamma  y    = (\overline{\psi})^{-1} b.
\end{array}\right.
\end{array}
\]
We make  the  following transformation: (I row + II  row $\rightarrow$ I row) and obtain the system:
$$\left\{
\begin{array} {l}
I{\varphi}^{-1}\beta  y + (\overline{\psi})^{-1} \gamma  y  =  I\varphi^{-1}a + (\overline{\psi})^{-1} b \\
x + (\overline{\psi})^{-1} \gamma  y    = (\overline{\psi})^{-1} b.
\end{array}
\right.
$$

Write expression $I{\varphi}^{-1}\beta  y + (\overline{\psi})^{-1} \gamma  y$ as follows $(-{\varphi}^{-1}\beta  + (\overline{\psi})^{-1} \gamma) y$.
Then the system (\ref{Right_ALIN__LEFT_LIN ORTH}) is equivalent to the following system
$$\left\{
\begin{array} {l}
(-{\varphi}^{-1}\beta  + (\overline{\psi})^{-1} \gamma) y =  I\varphi^{-1}a + (\overline{\psi})^{-1} b \\
x + (\overline{\psi})^{-1} \gamma  y    = (\overline{\psi})^{-1} b.
\end{array}
\right.
$$
It is clear that  the system (\ref{Right_ALIN__LEFT_LIN ORTH}) has a unique solution if and only if the mapping  $-{\varphi}^{-1}\beta  + (\overline{\psi})^{-1} \gamma $ is a permutation of the set $Q$.
We simplify the last equality.
\[
-{\varphi}^{-1}\beta  + (\overline{\psi})^{-1} \gamma = I \varphi^{-1}\beta  + I \psi^{-1} \gamma = I (\psi^{-1} \gamma + \varphi^{-1}\beta )
\]

Therefore the system (\ref{Right_ALIN__LEFT_LIN ORTH}) has a unique solution if and only if the mapping $(\psi^{-1} \gamma + \varphi^{-1}\beta )$ is a permutation of the set $Q$.
\end{proof}

\begin{theorem}\label{ONE_MORE_FORM} Left linear quasigroup $(Q,\cdot)$ and right  alinear quasigroup $(Q,\circ)$ of the form $x\cdot y = \varphi  y + \beta x$ and $x\circ y = \gamma  x + \overline{\psi} y$, respectively, which are defined over a
 group $(Q,+)$, are orthogonal if and only if the mapping   $(\psi^{-1} \gamma + \varphi^{-1}\beta )$ is a permutation of the set $Q$.
 \end{theorem}
\begin{proof}
The proof is similar to the proof of Theorem \ref{LEFT_LIN_Right_ALINEAR} and we omit it.
\end{proof}

\begin{theorem}\label{LEFT_LIN_LEFT_ALINEAR} Left linear quasigroup $(Q,\cdot)$ and left   alinear quasigroup $(Q,\circ)$ of the form $x\cdot y = \varphi  x + \beta y$ and $x\circ y = I\psi  x + \delta y$, respectively, which are defined over a
 group $(Q,+)$, are orthogonal if and only if the mapping  $(\psi^{-1} \delta  + J_{I\psi^{-1} b}{\varphi}^{-1}\beta  )$ is a  permutation of the set $Q$ for any $b \in Q$.
 \end{theorem}
\begin{proof}
Quasigroups $(Q, \cdot)$ and  $(Q,\circ)$ are orthogonal if and only if the system of equations
\begin{equation} \label{LEFT_ALIN__LEFT_LIN ORTH}
\left\{
\begin{split}
& \varphi  x + \beta y = a \\
& \overline{\psi} x + \delta  y = b
\end{split}
\right.
\end{equation}
 has a unique solution for any fixed elements $ a, b \in Q$.
We solve this system of equations in the usual way.
\[
\begin{array}{lll}
\left\{
\begin{array} {l}
x + {\varphi}^{-1}\beta  y   =  \varphi^{-1}a \\
(\overline{\psi})^{-1} \delta  y + x    = (\overline{\psi})^{-1} b
\end{array}
\right. & \Longleftrightarrow & \left\{
\begin{array} {l}
J_x {\varphi}^{-1}\beta  y + x   =  \varphi^{-1}a \\
Ix + \psi^{-1} \delta  y    = \psi^{-1} b
\end{array}\right.
\end{array}
\]
We make  the  following transformation: (I row + II  row $\rightarrow$ I row) and obtain the system:

\begin{equation} \label{LEFT_LIN_LEFT_ALIN}
\left\{
\begin{split}
& J_x {\varphi}^{-1}\beta  y + \psi^{-1} \delta  y = \varphi^{-1}a + \psi^{-1} b\\
& Ix    = \psi^{-1} b + I\psi^{-1} \delta  y
\end{split}
\right.
\end{equation}

We simplify the left part of the  first equation of system (\ref{LEFT_LIN_LEFT_ALIN}) using the second equation of this system:
\begin{equation*}
\begin{split}
& J_x {\varphi}^{-1}\beta  y + \psi^{-1} \delta  y = \\
& x +  {\varphi}^{-1}\beta  y - x + \psi^{-1} \delta  y = \\
& \psi^{-1} \delta  y - \psi^{-1} b  +  {\varphi}^{-1}\beta  y  + \psi^{-1} b + I\psi^{-1} \delta  y + \psi^{-1} \delta  y = \\
& \psi^{-1} \delta  y - \psi^{-1} b  +  {\varphi}^{-1}\beta  y + \psi^{-1} b  = \\
& \psi^{-1} \delta  y + J_{I\psi^{-1} b}{\varphi}^{-1}\beta  y.
\end{split}
\end{equation*}

Write expression $(\psi^{-1} \delta  y + J_{I\psi^{-1} b}{\varphi}^{-1}\beta  y)$ as follows $(\psi^{-1} \delta  + J_{I\psi^{-1} b}{\varphi}^{-1}\beta  )y $. Then the system (\ref{LEFT_LIN_LEFT_ALIN}) is equivalent to the following system
$$\left\{
\begin{array} {l}
(\psi^{-1} \delta  + J_{I\psi^{-1} b}{\varphi}^{-1}\beta  )y = \varphi^{-1}a + \psi^{-1} b\\
Ix    = \psi^{-1} b + I\psi^{-1} \delta  y.
\end{array}
\right.
$$
It is clear that  the system (\ref{LEFT_ALIN__LEFT_LIN ORTH}) has a unique solution if and only if the mapping  $(\psi^{-1} \delta  + J_{I\psi^{-1} b}{\varphi}^{-1}\beta  )$ is a  permutation of the set $Q$ for any $b \in Q$.
\end{proof}

\begin{theorem}\label{RIGHT_LIN_LEFT_ALINEAR_U} Left alinear quasigroup $(Q,\cdot)$ and right  linear quasigroup $(Q,\circ)$ of the form $x\cdot y = \overline{\varphi}  x + \beta y$ and $x\circ y = \gamma  y + \psi x$, respectively, which are defined over a
 group $(Q,+)$, are orthogonal if and only if the mapping   $(\psi^{-1} \gamma + \varphi^{-1} \beta)$ is a permutation of the set $Q$.
 \end{theorem}
\begin{proof}
Quasigroups $(Q, \cdot)$ and  $(Q,\circ)$ are orthogonal if and only if the system of equations
\begin{equation} \label{RIGHT_LIN_LEFT_ALINEAR}
\left\{
\begin{split}
& \overline{\varphi}  x + \beta y = a \\
& \gamma  y + \psi x  = b
\end{split}
\right.
\end{equation}
 has a unique solution for any fixed elements $a, b \in Q$.
We solve this system of equations in the usual way:
\[
\begin{array}{lll}
\left\{
\begin{array} {l}
 (\overline{\varphi})^{-1} \beta y + x  = (\overline{\varphi})^{-1} a \\
 \psi^{-1}\gamma  y +  x  = \psi^{-1} b
\end{array}
\right. & \Longleftrightarrow & \left\{
\begin{array} {l}
(\overline{\varphi})^{-1} \beta y + x  = (\overline{\varphi})^{-1} a \\
- x - \psi^{-1}\gamma  y   = - \psi^{-1} b.
\end{array}\right.
\end{array}
\]
We do  the  following transformation: (I row + II  row $\rightarrow$ I row) and obtain the system:
$$\left\{
\begin{array} {l}
(\overline{\varphi})^{-1} \beta y - \psi^{-1}\gamma  y  = (\overline{\varphi})^{-1} a - \psi^{-1} b \\
- x - \psi^{-1}\gamma  y   = - \psi^{-1} b.
\end{array}
\right.
$$

Write expression $(\overline{\varphi})^{-1} \beta y - \psi^{-1}\gamma  y $ as follows $((\overline{\varphi})^{-1} \beta  - \psi^{-1}\gamma)  y $.
Then the system (\ref{RIGHT_LIN_LEFT_ALINEAR}) is equivalent to the following system
$$\left\{
\begin{array} {l}
((\overline{\varphi})^{-1} \beta  - \psi^{-1}\gamma)  y  = (\overline{\varphi})^{-1} a - \psi^{-1} b \\
- x - \psi^{-1}\gamma  y   = - \psi^{-1} b.
\end{array}
\right.
$$
It is clear that  the system (\ref{RIGHT_LIN_LEFT_ALINEAR}) has a unique solution if and only if the mapping  $((\overline{\varphi})^{-1} \beta  - \psi^{-1}\gamma) = (I \varphi^{-1} \beta + I \psi^{-1} \gamma) = I(\psi^{-1} \gamma + \varphi^{-1} \beta)$ is a permutation of the set $Q$.
\end{proof}

Theorem 7 from \cite{SOKH_FRYZ_12} on conditions of orthogonality of linear quasigroups   follows from Theorems  \ref{Right_LINEAR_ORTHOGON} and \ref{LEFT_LINEAR_ORTHOGON}.

\begin{theorem}  \label{Linear_ORTHOGON_12} A linear quasigroup $(Q,\cdot)$ of the form $x\cdot y = \alpha  x + \beta y + c$
 and a linear quasigroup  $(Q,\circ)$ of the form $x\circ y = \gamma y + \delta x + d$, both defined over a
 group $(Q,+)$, are orthogonal if and only if the map $(- J_t \gamma^{-1} \delta  + \beta^{-1} \alpha )$ is a permutation of the set $Q$ for any element $t\in Q$.
\end{theorem}
\begin{proof}
Quasigroups $(Q, \cdot)$ and  $(Q,\circ)$ are orthogonal if and only if the system of equations
\begin{equation*}
\left\{
\begin{split}
& \alpha x + \beta  y + c = a \\
& \gamma  y + \delta  x +  d  = b
\end{split}
\right.
\end{equation*}
 has a unique solution for any fixed elements $a, b \in Q$.

We solve this system of equations as follows:
\[
\begin{array}{lll}
\left\{
\begin{array} {l}
\alpha x + \beta  y  = a - c \\
\gamma  y + \delta  x   = b- d
\end{array}
\right. & \Longleftrightarrow & \left\{
\begin{array} {l}
\beta^{-1} \alpha x +  y  = \beta^{-1}(a - c) \\
J_{\gamma  y} \delta  x + \gamma y = (b- d),
\end{array}\right.
\end{array}
\]
where $J_{\gamma  y} \delta  x = \gamma y + \delta x - \gamma y$. Notice  $\gamma^{-1} J_{\gamma  y} \delta  x =  J_{y} \gamma^{-1} \delta  x$.

Further we have:
\begin{equation} \label{LIN_QUAS_ORTH_12}
\left\{
\begin{split}
& \beta^{-1} \alpha x +  y  = \beta^{-1}(a - c) \\
& - y - \gamma^{-1} J_{\gamma  y} \delta  x  = -\gamma^{-1}(b - d).
\end{split}
\right.
\end{equation}
If in the    system (\ref{LIN_QUAS_ORTH_12}) we add the first  and the second  row and write the sum instead of second row ($I + II \rightarrow II$), then we obtain the following system
\begin{equation} \label{TWO_LIN_QUAS_ORTH_12}
\left\{
\begin{split}
& \beta^{-1} \alpha x +  y  = \beta^{-1}(a - c) \\
& \beta^{-1} \alpha x - \gamma^{-1} J_{\gamma  y} \delta  x  = \beta^{-1}(a - c) -\gamma^{-1}(b - d).
\end{split}
\right.
\end{equation}

Therefore we can rewrite the system   (\ref{TWO_LIN_QUAS_ORTH_12}) in the following form
\begin{equation} \label{TWO_LIN_QUAS_ORTH_123}
\left\{
\begin{split}
&   y  = - \beta^{-1} \alpha x + \beta^{-1}(a - c) \\
& \beta^{-1} \alpha x - J_{y} \gamma^{-1} \delta  x  = \beta^{-1}(a - c) -\gamma^{-1}(b - d).
\end{split}
\right.
\end{equation}

Rewrite the left part of the second equation of the system (\ref{TWO_LIN_QUAS_ORTH_123}) in the following form
\begin{equation*}
\beta^{-1} \alpha x + IJ_{y}\gamma^{-1} \delta  x  = \beta^{-1} \alpha x +  J_{y}I\gamma^{-1} \delta  x =  \beta^{-1} \alpha x + y -  \gamma^{-1} \delta  x - y.
\end{equation*}

Further, taking into
consideration first equation of the system  (\ref{TWO_LIN_QUAS_ORTH_123}), we have:
\begin{equation*}
\begin{split}
& \beta^{-1} \alpha x + y -  \gamma^{-1} \delta  x - y = \\
&\beta^{-1} \alpha x - \beta^{-1} \alpha x + \beta^{-1}(a - c) -  \gamma^{-1} \delta  x - \beta^{-1}(a - c) + \beta^{-1} \alpha x = \\
& \beta^{-1}(a - c) -  \gamma^{-1} \delta  x - \beta^{-1}(a - c) + \beta^{-1} \alpha x = \\
& J_{\beta^{-1}(a - c)}I\gamma^{-1} \delta  x + \beta^{-1} \alpha x = - J_{\beta^{-1}(a - c)}\gamma^{-1} \delta  x + \beta^{-1} \alpha x.
\end{split}
\end{equation*}
Similarly, as in Theorem \ref{LEFT_LINEAR_ORTHOGON}, we write expression $- J_{\beta^{-1}(a - c)}\gamma^{-1} \delta  x + \beta^{-1} \alpha x$ in the following form  $ (- J_{\beta^{-1}(a - c)}\gamma^{-1} \delta  + \beta^{-1} \alpha ) x $.
The system (\ref{TWO_LIN_QUAS_ORTH_123}) takes the form
\begin{equation} \label{TWO_LIN_QUAS_ORTH_1234}
\left\{
\begin{split}
&   y  = - \beta^{-1} \alpha x + \beta^{-1}(a - c) \\
& (- J_{\beta^{-1}(a - c)}\gamma^{-1} \delta  + \beta^{-1} \alpha ) x  = \beta^{-1}(a - c) -\gamma^{-1}(b - d).
\end{split}
\right.
\end{equation}

From the system (\ref{TWO_LIN_QUAS_ORTH_1234}) it follows that quasigroups $(Q,\cdot)$ and $(Q,\circ)$ are orthogonal if and only if the map
$(- J_{\beta^{-1}(a - c)}\gamma^{-1} \delta  + \beta^{-1} \alpha )$ is a  permutation of the set $Q$ for any element $a\in Q$.

Denote the expression $\beta^{-1}(a - c)$ by the letter $t$. We can reformulate the last condition as follows: quasigroups $(Q, \cdot)$ and $(Q, \circ)$ are orthogonal if and only if the map
$(- J_t \gamma^{-1} \delta  + \beta^{-1} \alpha )$ is a permutation of the set $Q$ for any element $t\in Q$.
\end{proof}

Taking into consideration that we have not proved an analogue of Lemma \ref{FORMS_LEFT_LIN_QUAS} for linear quasigroup, we give independent from Theorems  \ref{Right_LINEAR_ORTHOGON} and \ref{LEFT_LINEAR_ORTHOGON} proof of the following

\begin{theorem} \cite{SOKH_FRYZ_12}. \label{Linear_ORTHOGON} A linear quasigroup $(Q,\cdot)$ of the form $x\cdot y = \alpha  x + \beta y + c$
 and a linear quasigroup  $(Q,\circ)$ of the form $x\circ y = \gamma x + \delta y + d$, both defined over a
 group $(Q,+)$, are orthogonal if and only if the map $(- \gamma^{-1}\delta + \alpha^{-1}\beta )$ is a permutation of the set $Q$.
\end{theorem}
\begin{proof}
Quasigroups $(Q, \cdot)$ and  $(Q,\circ)$ are orthogonal if and only if the system of equations
\begin{equation*}
\left\{
\begin{split}
& \alpha x + \beta  y + c = a \\
& \gamma  x + \delta  y +  d  = b
\end{split}
\right.
\end{equation*}
 has a unique solution for any fixed elements $a, b \in Q$.

We solve this system of equations as follows:
\[
\begin{array}{lll}
\left\{
\begin{array} {l}
\alpha x + \beta  y  = a - c \\
\gamma  x + \delta  y   = b- d
\end{array}
\right. & \Longleftrightarrow & \left\{
\begin{array} {l}
x + \alpha^{-1}\beta  y  = \alpha^{-1}(a - c) \\
 - \gamma^{-1}\delta  y - x  = -\gamma^{-1}(b- d).
\end{array}\right.
\end{array}
\]

In the last system we add the second  and the first  row and write the sum instead of second row ($II + I \rightarrow II$). We obtain the following system
\begin{equation} \label{T_QUAS_ORTH}
\left\{
\begin{split}
& x + \alpha^{-1}\beta  y  = \alpha^{-1}(a - c) \\
&  - \gamma^{-1}\delta y + \alpha^{-1}\beta   y   = -\gamma^{-1}(b - d) + \alpha^{-1}(a - c).
\end{split}
\right.
\end{equation}

Similarly as in Theorem \ref{LEFT_LINEAR_ORTHOGON} we write expression $- \gamma^{-1}\delta y + \alpha^{-1}\beta   y $ in the following form  $ (- \gamma^{-1}\delta  + \alpha^{-1}\beta ) y $.
From the system (\ref{T_QUAS_ORTH}) it follows that quasigroups $(Q,\cdot)$ and $(Q,\circ)$ are orthogonal if and only if the map
$(- \gamma^{-1}\delta + \alpha^{-1}\beta )$ is a permutation of the set $Q$.
\end{proof}

\begin{theorem}\label{T_ORTHOGON} A $T$-quasigroup $(Q,\cdot)$ of the form $x\cdot y = \alpha  x + \beta y + c$
 and a $T$-quasigroup  $(Q,\circ)$ of the form $x\circ y = \gamma x + \delta y + d$, both defined over a
 group $(Q,+)$, are orthogonal if and only if the map $\alpha^{-1}\beta - \gamma^{-1}\delta$ is an automorphism
of the group $(Q,+)$ \cite[Theorem 16]{MS1}.
\end{theorem}
\begin{proof}
The proof follows from Theorem \ref{Linear_ORTHOGON} and the fact that in abelian group $- \gamma^{-1}\delta + \alpha^{-1}\beta = \alpha^{-1}\beta  - \gamma^{-1}\delta$ and that  the map $\alpha^{-1}\beta  - \gamma^{-1}\delta$ is an endomorphism of the group $(Q, +)$.
\end{proof}

\begin{lemma} \label{psi_varphi_Permut}
If $(Q, +)$ is an abelian group, $\varphi, \psi \in Aut(Q, +)$, then $\varphi - \psi$  is an automorphism of the group $(Q, +)$ if and only if $\psi - \varphi$ is an automorphism of this group.
\end{lemma}
\begin{proof}
 Taking into consideration that the map $I(x) = -x$ is an automorphism of an abelian group $(Q, +)$ and a permutation of the set $Q$ of order two, we have $-(\varphi - \psi) = -\varphi +  \psi  = \psi - \varphi$.
\end{proof}

\begin{lemma} \label{LEMMA_ORTHOGON}
In conditions of  Theorem \ref{T_ORTHOGON} the following  statements    are equivalent:
"the endomorphism $(\alpha^{-1}\beta - \gamma^{-1}\delta)$ is an automorphism of $(Q, +)$" and "the endomorphism $(\beta^{-1}\alpha -\delta^{-1} \gamma)$ is an automorphism of $(Q, +)$".
\end{lemma}
\begin{proof}
The map $(\alpha^{-1}\beta - \gamma^{-1}\delta)$ is a permutation of the set $Q$ if and only if the map
$\varepsilon - \alpha\gamma^{-1} \delta \beta^{-1}$ is a permutation of the set $Q$. Indeed,
$\alpha(\alpha^{-1}\beta - \gamma^{-1}\delta)\beta = \varepsilon - \alpha\gamma^{-1} \delta \beta^{-1}$.

Similarly, $(\beta^{-1}\alpha -\delta^{-1} \gamma)$ is a permutation of set $Q$ if and only if the map
\begin{equation} \label{ENDOM_IS_AN_AUTOM}
\varepsilon - \beta \delta^{-1} \gamma \alpha^{-1}
\end{equation}
 is a permutation of the set $Q$.

If we denote the map
$\alpha \gamma^{-1} \delta \beta^{-1}$ by $\psi$, then $\beta \delta^{-1} \gamma \alpha^{-1} = \psi^{-1}$.

Further we have the following equivalence:  the map $\varepsilon - \psi$ is a permutation if and only if the map
$\varepsilon - \psi^{-1}$  is a permutation of the set $Q$.

Indeed, $\varepsilon - \psi$ is a permutation if and
only if the map $\psi - \varepsilon$ is a permutation (Lemma \ref{psi_varphi_Permut}), further   $\psi - \varepsilon$ is a permutation if and only if
$\psi^{-1}(\psi - \varepsilon) = \varepsilon - \psi^{-1}$  is a permutation.
\end{proof}

\begin{corollary} \label{4.1} A $T$-quasigroup $(Q,\cdot)$ of the form
$x\cdot y = \varphi  x + \psi y + c$ over a  group $(Q,+)$ and its $(12)$-parastrophe  $(Q,\star)$ of the form
$x\star y = \psi x + \varphi y + c$ are orthogonal if and only if the map $\varphi ^{-1}\psi - \psi^{-1}
\varphi$ is an automorphism of the group $(Q,+)$ .
\end{corollary}

\begin{corollary}  A $T$-quasigroup $(Q,\cdot)$ of the form $x\cdot y = \alpha  x + \beta y + c$
 and a medial quasigroup  $(Q,\circ)$ of the form $x\circ y = \gamma x + \delta y + d$,
  both over a group $(Q,+)$, are orthogonal if and only if the map $\alpha\delta  - \gamma\beta$ is an automorphism
of the group $(Q,+)$.
\end{corollary}
\begin{proof}
From equality (\ref{ENDOM_IS_AN_AUTOM})  it follows  that  quasigroups $(Q,\cdot)$ and $(Q,\circ)$ are
orthogonal if and only if the map $\varepsilon - \beta \delta^{-1} \gamma \alpha^{-1}$ is a permutation of the
set $Q$. Further, since $\delta\gamma=  \gamma \delta,$ we have $\beta \delta^{-1} \gamma \alpha^{-1} = \beta
\gamma \delta^{-1}  \alpha^{-1}$ and the map $\varepsilon - \beta \delta^{-1} \gamma \alpha^{-1}$ is a
permutation of the set $Q$ if and only if the map $(\varepsilon - \beta \gamma \delta^{-1}  \alpha^{-1})\alpha
\delta  = \alpha \delta - \beta \gamma$ is a permutation of the set $Q$.
\end{proof}

\begin{theorem}  \label{ALinear_ORTHOGON_12} Alinear quasigroup $(Q,\cdot)$ of the form $x\cdot y = I\alpha  x + I\beta y + c$
 and  alinear quasigroup  $(Q,\circ)$ of the form $x\circ y = I\gamma y + I\delta x + d$, both defined over a
 group $(Q,+)$, where $\alpha, \beta, \gamma, \delta \in Aut(Q, +)$, are orthogonal if and only if the map $(\beta^{-1} \alpha  - J_t\gamma^{-1} \delta)$ is a permutation of the set $Q$ for any element $t\in Q$.
\end{theorem}
\begin{proof}
Quasigroups $(Q, \cdot)$ and  $(Q,\circ)$ are orthogonal if and only if the system of equations
\begin{equation*}
\left\{
\begin{split}
& I\alpha x + I\beta  y + c = a \\
& I\gamma  y + I\delta  x +  d  = b
\end{split}
\right.
\end{equation*}
 has a unique solution for any fixed elements $a, b \in Q$.

We solve this system of equations as follows:
\[
\begin{array}{lll}
\left\{
\begin{array} {l}
I\alpha x + I\beta  y  = a - c \\
I\gamma  y + I\delta  x   = b- d
\end{array}
\right. & \Longleftrightarrow & \left\{
\begin{array} {l}
y + \beta^{-1} \alpha x  = I\beta^{-1}(a - c) \\
J_{I\gamma  y} I\delta  x + I\gamma y = (b- d),
\end{array}\right.
\end{array}
\]
where $J_{I\gamma  y} I\delta  x = I\gamma y + I\delta x - I\gamma y$. Notice
$\gamma^{-1} J_{I\gamma  y} I\delta    x =  J_{Iy} \gamma^{-1} I\delta  x$.

Further we have:
\begin{equation} \label{ALIN_QUAS_ORTH_12}
\left\{
\begin{split}
& y + \beta^{-1} \alpha x  = I\beta^{-1}(a - c) \\
& J_{I y} \gamma^{-1}  I\delta  x + Iy = \gamma^{-1} (b- d).
\end{split}
\right.
\end{equation}
If in the    system (\ref{ALIN_QUAS_ORTH_12}) we add the second   and the first  row and write the sum instead of the second row ($II + I \rightarrow II$), then we obtain the following system
\begin{equation} \label{TWO_ALIN_QUAS_ORTH_12}
\left\{
\begin{split}
& y + \beta^{-1} \alpha x  = I\beta^{-1}(a - c) \\
& J_{I y} \gamma^{-1}  I\delta  x + \beta^{-1} \alpha x = \gamma^{-1} (b- d)- \beta^{-1}(a - c).
\end{split}
\right.
\end{equation}

Therefore we can rewrite the system   (\ref{TWO_ALIN_QUAS_ORTH_12}) in the following form
\begin{equation} \label{TWO_ALIN_QUAS_ORTH_123}
\left\{
\begin{split}
&  y    = I\beta^{-1}(a - c) + I\beta^{-1} \alpha x \\
& J_{I y} \gamma^{-1}  I\delta  x + \beta^{-1} \alpha x = \gamma^{-1} (b- d)- \beta^{-1}(a - c).
\end{split}
\right.
\end{equation}

Rewrite the left part of the second equation of the system (\ref{TWO_ALIN_QUAS_ORTH_123}) in the following form
\begin{equation*}
\begin{split}
& J_{I y} \gamma^{-1}  I\delta  x + \beta^{-1} \alpha x =
 -(J_{I y} \gamma^{-1} \delta  x) +  \beta^{-1} \alpha x = \\
 & -( - y + \gamma^{-1} \delta  x + y) +  \beta^{-1} \alpha x= \\
 & - y - \gamma^{-1} \delta  x +  y +  \beta^{-1} \alpha x \overset{(\ref{TWO_ALIN_QUAS_ORTH_123})}{=} \\
 & - (I\beta^{-1}(a - c) + I\beta^{-1} \alpha x) - \gamma^{-1} \delta  x +  I\beta^{-1}(a - c) + I\beta^{-1} \alpha x +  \beta^{-1} \alpha x = \\
 & \beta^{-1} \alpha x + \beta^{-1}(a - c)  - \gamma^{-1} \delta  x - \beta^{-1}(a - c) = \\
& \beta^{-1} \alpha x - J_{\beta^{-1}(a - c)} \gamma^{-1} \delta  x.
\end{split}
\end{equation*}

Similarly, as in Theorem \ref{LEFT_LINEAR_ORTHOGON}, we write expression $\beta^{-1} \alpha x - J_{\beta^{-1}(a - c)} \gamma^{-1} \delta  x$ in the following form  $(\beta^{-1} \alpha  - J_{\beta^{-1}(a - c)} \gamma^{-1} \delta) x$.
The system (\ref{TWO_ALIN_QUAS_ORTH_123}) takes the form
\begin{equation} \label{TWO_ALIN_QUAS_ORTH_1234}
\left\{
\begin{split}
&  y    = I\beta^{-1}(a - c) + I\beta^{-1} \alpha x \\
& (\beta^{-1} \alpha  - J_{\beta^{-1}(a - c)} \gamma^{-1} \delta) x = \gamma^{-1} (b- d)- \beta^{-1}(a - c).
\end{split}
\right.
\end{equation}

From  system (\ref{TWO_ALIN_QUAS_ORTH_1234}) it follows that quasigroups $(Q,\cdot)$ and $(Q,\circ)$ are orthogonal if and only if the map
$(\beta^{-1} \alpha  - J_{\beta^{-1}(a - c)} \gamma^{-1} \delta)$ is a  permutation of the set $Q$ for any element $a\in Q$.

Denote the expression $\beta^{-1}(a - c)$ by the letter $t$. We can reformulate the last condition as follows: quasigroups $(Q, \cdot)$ and $(Q, \circ)$ are orthogonal if and only if the map
$(\beta^{-1} \alpha  - J_t\gamma^{-1} \delta)$ is a permutation of the set $Q$ for any element $t\in Q$.
\end{proof}

\begin{theorem}  \label{Linear_LIN_ORTHOGON_12} Left linear quasigroup $(Q,\cdot)$ of the form $x\cdot y = \varphi  x + \beta y$
 and right linear quasigroup  $(Q,\circ)$ of the form $x\circ y = \gamma y + \psi  x$, both defined over a
 group $(Q,+)$, where $\varphi, \psi \in Aut(Q, +)$, are orthogonal if and only if the map $(J_{t} \psi^{-1}\gamma    - \varphi^{-1}\beta)$ is a permutation of the set $Q$ for any element $t\in Q$.
\end{theorem}
\begin{proof}
Quasigroups $(Q, \cdot)$ and  $(Q,\circ)$ are orthogonal if and only if the system of equations
\begin{equation*}
\left\{
\begin{split}
& \varphi x + \beta y  = a \\
& \gamma  y + \psi  x   = b
\end{split}
\right.
\end{equation*}
 has a unique solution for any fixed elements $a, b \in Q$.

We solve this system of equations as follows:
\[
\begin{array}{lll}
\left\{
\begin{array} {l}
I\varphi^{-1}\beta y + I x  = I\varphi^{-1}a  \\
\psi^{-1} \gamma  y +  x   = \psi^{-1} b
\end{array}
\right. & \Longleftrightarrow & \left\{
\begin{array} {l}
I\varphi^{-1}\beta y + I x  = I\varphi^{-1}a  \\
x + J_{-x}\psi^{-1} \gamma  y    = \psi^{-1} b
\end{array}\right.
\end{array}
\]
where $ J_{-x} \gamma  y = -x + \gamma y + x$.

If in the    last system  we add the first    and the second equation  and write the sum instead of the second equation  ($I + II \rightarrow II$), then we obtain the following system
\begin{equation} \label{TWO_LINa_LIN_QUAS_ORTH_12}
\left\{
\begin{split}
& I\varphi^{-1}\beta y + I x  = I\varphi^{-1}a  \\
& I\varphi^{-1}\beta y + J_{-x} \psi^{-1}\gamma  y = I\varphi^{-1}a + \psi^{-1} b.
\end{split}
\right.
\end{equation}

Therefore we can rewrite  system   (\ref{TWO_LINa_LIN_QUAS_ORTH_12}) in the following form
\begin{equation} \label{TWO_LIN_LIN_QUAS_ORTH_123}
\left\{
\begin{split}
&  x    = \varphi^{-1} a + I \varphi^{-1}\beta y  \\
& I\varphi^{-1}\beta y + J_{-x} \psi^{-1}\gamma  y = I\varphi^{-1}a + \psi^{-1} b.
\end{split}
\right.
\end{equation}

Rewrite the left part of the second equation of  system (\ref{TWO_LIN_LIN_QUAS_ORTH_123}) in the following form
\begin{equation*}
\begin{split}
&  I\varphi^{-1}\beta y + J_{-x} \psi^{-1}\gamma  y = \\
& I\varphi^{-1}\beta y - x + \psi^{-1}\gamma  y + x = \\
& I\varphi^{-1}\beta y  + \varphi^{-1}\beta y - \varphi^{-1} a +  \psi^{-1}\gamma  y + \varphi^{-1} a + I \varphi^{-1}\beta y = \\
&  - \varphi^{-1} a +  \psi^{-1}\gamma  y + \varphi^{-1} a + I \varphi^{-1}\beta y = \\
&  J_{I\varphi^{-1} a} \psi^{-1}\gamma  y  - \varphi^{-1}\beta y.
\end{split}
\end{equation*}

We write expression $J_{I\varphi^{-1} a} \psi^{-1}\gamma  y  - \varphi^{-1}\beta y$ in the following form  $(J_{I\varphi^{-1} a} \psi^{-1}\gamma    - \varphi^{-1}\beta ) y$.
The system (\ref{TWO_LIN_LIN_QUAS_ORTH_123}) takes the form
\begin{equation} \label{TWO_LIN_LIN_QUAS_ORTH_1234}
\left\{
\begin{split}
& x    = \varphi^{-1} a + I \varphi^{-1}\beta y  \\
& (J_{I\varphi^{-1} a} \psi^{-1}\gamma    - \varphi^{-1}\beta ) y = I\varphi^{-1}a + \psi^{-1} b.
\end{split}
\right.
\end{equation}

From  system (\ref{TWO_LIN_LIN_QUAS_ORTH_1234}) it follows that quasigroups $(Q,\cdot)$ and $(Q,\circ)$ are orthogonal if and only if the map
$(J_{I\varphi^{-1} a} \psi^{-1}\gamma    - \varphi^{-1}\beta)$ is a  permutation of the set $Q$ for any element $a\in Q$.

Denote  expression $I\varphi^{-1} a$ by the letter $t$. We can reformulate the last condition as follows: quasigroups $(Q, \cdot)$ and $(Q, \circ)$ are orthogonal if and only if the map
$(J_{t} \psi^{-1}\gamma    - \varphi^{-1}\beta)$ is a permutation of the set $Q$ for any element $t\in Q$.
\end{proof}

\begin{theorem}  \label{Linear_LIN_ORTHOGON_012} Left linear quasigroup $(Q,\cdot)$ of the form $x\cdot y = \varphi  y + \beta x$
 and right linear quasigroup  $(Q,\circ)$ of the form $x\circ y = \gamma x + \psi  y$, both defined over a
 group $(Q,+)$, where $\varphi, \psi \in Aut(Q, +)$, are orthogonal if and only if the map $(I \psi^{-1}\gamma    + J_{I\psi^{-1}b} \varphi^{-1}\beta)$ is a permutation of the set $Q$ for any element $b\in Q$.
\end{theorem}
\begin{proof}
The proof is similar to the proof of Theorem \ref{Linear_LIN_ORTHOGON_12} and we omit it.
\end{proof}

\begin{corollary}\label{COROLL_LIN_PAR}
If in  conditions of Theorem  \ref{Linear_ORTHOGON_12} (Theorem  \ref{ALinear_ORTHOGON_12})  the group $Inn(Q,+)$ of inner automorphisms of the group $(Q, +)$ acts on the group $Aut(Q, +)$ transitively, then does not exist orthogonal quasigroups $(Q, \cdot)$ and $(Q, \circ)$.
\end{corollary}
\begin{proof}
Since the group $Inn(Q,+)$ acts transitively, then in conditions of Theorem \ref{Linear_ORTHOGON_12} there exists an element $s \in Q$ such that $J_s \gamma^{-1} \delta  =  \beta^{-1} \alpha$, i.e.,  such that  $(- J_s \gamma^{-1} \delta  + \beta^{-1} \alpha) x = 0$ for any $x\in Q$.

In conditions of Theorem \ref{ALinear_ORTHOGON_12} there exists an element $d \in Q$ such that
$(\beta^{-1} \alpha  - J_d\gamma^{-1} \delta)x = 0$ for any $x\in Q$.
\end{proof}

\begin{lemma}\label{Lemma_EXAMPLE}
\begin{enumerate}
\item
If in  conditions of Theorem   \ref{Linear_ORTHOGON_12}  the group $(Q, +)$ is  symmetric group $S_n$ ($n\neq 2; 6$), then does not exist orthogonal quasigroups $(S_n, \cdot)$ and $(S_n, \circ)$.
\item
If in  conditions of Theorem   \ref{ALinear_ORTHOGON_12}  the group $(Q, +)$ is  symmetric group $S_n$ ($n\neq 2; 6$), then does not exist orthogonal quasigroups $(S_n, \cdot)$ and $(S_n, \circ)$.
\end{enumerate}
\end{lemma}
\begin{proof}
By  G\"older theorem  $Aut(S_n) = Inn(S_n)$ for any natural number $n$, $n\neq 2; 6$  [p. 67]\cite{KM}.
\end{proof}

\section{Orthogonality of parastrophes of left(right) linear(alinear) quasigroups}

\begin{theorem} \label{PARASTR_ORTH_OF_LIN_Q}
For a linear quasigroup $(Q,A)$ of the form $A(x, y) =  \varphi x + \psi y + c$ over a group $(Q, +)$ the
following equivalences are fulfilled:
\begin{enumerate}
\item $A \bot A^{12} \Longleftrightarrow $  the map $(-J_t \varphi^{-1} \psi + \psi^{-1} \varphi)$ is a permutation of the set $Q$ for any $t\in Q$;

\item $A\bot A^{13} \Longleftrightarrow $ the map $(\varphi J_{I\varphi^{-1} c} + \varepsilon)$ is a permutation of the set $Q$;

\item $A\bot A^{23} \Longleftrightarrow $ the map $(\varepsilon + \psi)$ is a permutation of the set $Q$;

\item $A\bot A^{123} \Longleftrightarrow $ the map $(\varphi J^{-1}_{\psi^{-1} c} + \psi^2)$ is a permutation of the set $Q$;

\item $A\bot A^{132} \Longleftrightarrow $ the map $(\varphi^2 + \psi)$ is a permutation of the set $Q$.
\end{enumerate}
\end{theorem}
\begin{proof}
The forms of parastrophes of quasigroup $(Q, A)$ are given in Lemma \ref{Forms_of_parastrophes_OF_LIN_QUAS}.

Case 1. The proof follows from Theorem \ref{Linear_ORTHOGON_12}.

Case 2. Using  Theorem \ref{LEFT_LINEAR_ORTHOGON} we have: $A\bot A^{13}$ if and only if the map $I\varphi^{-1}\psi + \varphi I J_{I\varphi^{-1} c} \varphi^{-1} \psi$ is a permutation of the set $Q$.

    We make the following transformations: $I\varphi^{-1}\psi + \varphi I J_{I\varphi^{-1} c} \varphi^{-1} \psi = (I + \varphi I J_{I\varphi^{-1} c}) \varphi^{-1} \psi = (I + \varphi J_{I\varphi^{-1} c}I) \varphi^{-1} \psi = (\varphi J_{I\varphi^{-1} c} + \varepsilon) I\varphi^{-1} \psi$.
    The last map is a permutation if and only if the map  $(\varphi J_{I\varphi^{-1} c} + \varepsilon)$ is a permutation of the set $Q$.

Case 3.  Using  Theorem \ref{Right_LINEAR_ORTHOGON}  we have: $A\bot A^{23}$ if and only if the map $\psi^{-1}\varphi - \psi I \psi^{-1} \varphi$ is a permutation of the set $Q$. We simplify the last equality in the following way:
\[
\psi^{-1}\varphi - \psi I \psi^{-1} \varphi = \psi^{-1}\varphi + \psi  \psi^{-1} \varphi = (\varepsilon + \psi)  \psi^{-1}\varphi.
\]
Therefore $A\bot A^{23}$ if and only if the map $(\varepsilon + \psi)$ is a permutation of the set $Q$.

Case 4. From Theorem \ref{LEFT_LIN_Right_ALINEAR} it follows that $A\bot A^{123}$ if and only if the map
\begin{equation}\label{123_LIN_PAR}
I\varphi^{-1}\psi + (I J_{\varphi^{-1}c} \varphi^{-1}\psi)^{-1}\varphi^{-1}
\end{equation}
is a permutation of the set $Q$. We make the following transformation of  expression (\ref{123_LIN_PAR}):
\begin{equation}
\begin{split}
& I\varphi^{-1}\psi + (I J_{\varphi^{-1}c} \varphi^{-1}\psi)^{-1}\varphi^{-1} = \\
& I\varphi^{-1}\psi +   \psi^{-1} \varphi J^{-1}_{\varphi^{-1}c} I \varphi^{-1} = \\
& I(\psi^{-1} \varphi J^{-1}_{\varphi^{-1}c} \varphi^{-1} + \varphi^{-1}\psi) = \\
& I(\psi^{-1}  J^{-1}_{c}\varphi \varphi^{-1} + \varphi^{-1}\psi) = \\
& I(\psi^{-1}  J^{-1}_{c} + \varphi^{-1}\psi) = \\
& I\varphi^{-1} (\varphi \psi^{-1}  J^{-1}_{c} + \psi) = \\
& I\varphi^{-1} (\varphi J^{-1}_{\psi^{-1} c} \psi^{-1} + \psi) = \\
& I\varphi^{-1} (\varphi J^{-1}_{\psi^{-1} c} + \psi^2)\psi^{-1}.
 \end{split}
\end{equation}
We obtain:  $A\bot A^{123}$ if and only if the map $(\varphi J^{-1}_{\psi^{-1} c} + \psi^2)$ is a permutation of the set $Q$.

Case 5. From Theorem \ref{RIGHT_LIN_LEFT_ALINEAR_U} we have: $A\bot A^{132}$  if and only if the map $(I \psi^{-1} \varphi)^{-1}\psi^{-1} - \psi^{-1} \varphi $ is a permutation of the set $Q$. We simplify the last equality in the following way:
\begin{equation*}
\begin{split}
& (I \psi^{-1} \varphi)^{-1}\psi^{-1} - \psi^{-1} \varphi = \\
& \varphi^{-1} \psi I \psi^{-1}  + I \psi^{-1}\varphi =\\
& I \varphi^{-1}  +I \psi^{-1} \varphi = \\
& I(\psi^{-1} \varphi + \varphi^{-1})= \\
& I\psi^{-1}(\varphi^2 + \psi)\varphi^{-1}.
\end{split}
\end{equation*}
Therefore $A\bot A^{132}$ if and only if the map $(\varphi^2 + \psi)$ is a permutation of the set $Q$.
\end{proof}

Taking into consideration Lemma \ref{UP_TO_THIRD_COPMONENT} we can take in formulation of Theorem  \ref{PARASTR_ORTH_OF_LIN_Q} $c=0$ without loss of generality.
Therefore we can reformulate Theorem \ref{PARASTR_ORTH_OF_LIN_Q} in the following form:
\begin{theorem} \label{REFORMULTION}
For a linear quasigroup $(Q,A)$ of the form $A(x, y) =  \varphi x + \psi y + c$ over a group $(Q, +)$ the
following equivalences are fulfilled:
\begin{enumerate}
\item $A \bot A^{12} \Longleftrightarrow $  the map $(-J_t \varphi^{-1} \psi + \psi^{-1} \varphi)$ is a permutation of the set $Q$ for any $t\in Q$;

\item $A\bot A^{13} \Longleftrightarrow $ the map $(\varphi  + \varepsilon)$ is a permutation of the set $Q$;

\item $A\bot A^{23} \Longleftrightarrow $ the map $(\varepsilon + \psi)$ is a permutation of the set $Q$;

\item $A\bot A^{123} \Longleftrightarrow $ the map $(\varphi  + \psi^2)$ is a permutation of the set $Q$;

\item $A\bot A^{132} \Longleftrightarrow $ the map $(\varphi^2 + \psi)$ is a permutation of the set $Q$.
\end{enumerate}
\end{theorem}

\begin{corollary}
Any linear quasigroup over the  group $S_n$ ($n\neq 2; 6$) is not orthogonal to its $(12)$-parastrophe.
\end{corollary}
\begin{proof}
The proof follows from Theorem \ref{PARASTR_ORTH_OF_LIN_Q} and Lemma \ref{Lemma_EXAMPLE}.
\end{proof}

From  Theorem \ref{PARASTR_ORTH_OF_LIN_Q} it follows
\begin{corollary} \label{T1} \cite[Theorem 17]{MS05}.
For a $T$-quasigroup $(Q,A)$ of the form $A(x, y) =  \varphi x + \psi y + a$ over an abelian group $(Q, +)$ the
following equivalences are fulfilled:

(i) $A \bot A^{12} \Longleftrightarrow (\varphi - \psi), (\varphi + \psi)$ are permutations of the set $Q$;

(ii) $A\bot A^{13} \Longleftrightarrow (\varepsilon + \varphi)$ is a permutation of the set $Q$;

(iii) $A\bot A^{23} \Longleftrightarrow (\varepsilon + \psi)$ is a permutation of the set $Q$;

(iv) $A\bot A^{123} \Longleftrightarrow (\varphi+ \psi^2)$ is a permutation of the set $Q$;

(v) $A\bot A^{132} \Longleftrightarrow (\varphi^2 + \psi)$ is a permutation of the set $Q$.
\end{corollary}

\begin{theorem} \label{PARASTR_ORTH_OF_ALIN_Q}
For an alinear quasigroup $(Q,A)$ of the form $A(x, y) =  I\varphi x + I\psi y + c$ over a group $(Q, +)$ the
following equivalences are fulfilled:
\begin{enumerate}
\item $A \bot A^{12} \Longleftrightarrow $  the map $(\psi^{-1} \varphi - J_t \varphi^{-1} \psi)$  is a permutation of the set $Q$ for any $t\in Q$;

\item $A\bot A^{13} \Longleftrightarrow $ the map $(\varphi  - J_{\psi t}  J_{c})$ is a permutation of the set $Q$ for any $t\in Q$;

\item $A\bot A^{23} \Longleftrightarrow $ the map $(\varepsilon + I\psi J_t)$ is  a permutation of the set $Q$ for any $t\in Q$;

\item $A\bot A^{123} \Longleftrightarrow $ the map $(\psi^2 - \varphi  J_{\psi^{-1}c})$ is a permutation of the set $Q$;

\item $A\bot A^{132} \Longleftrightarrow $ the map $(\psi - \varphi^2)$  is a permutation of the set $Q$.
\end{enumerate}
\end{theorem}
\begin{proof}
The forms of parastrophes of quasigroup $(Q, A)$ are given in Lemma \ref{Forms_of_parastrophes_OF_ALIN_QUAS}.

Case 1. The proof follows from Theorem \ref{ALinear_ORTHOGON_12}.

Case 2. Using  Theorem \ref{ALinear_ORTHOGON_12} we have: $A\bot A^{13}$ if and only if the maps $\psi^{-1}\varphi  - J_t\psi^{-1} \varphi  J_{\varphi^{-1} c} \varphi^{-1}$ are permutation of the set $Q$.

    We make the following transformations: $\psi^{-1}\varphi  - J_t\psi^{-1} \varphi  J_{\varphi^{-1} c} \varphi^{-1} = \psi^{-1}\varphi  - \psi^{-1}J_{\psi t}  J_{c} \varphi \varphi^{-1} = \psi^{-1}\varphi  - \psi^{-1}J_{\psi t}  J_{c} = \psi^{-1} (\varphi  - J_{\psi t}  J_{c})$.
    The last maps are permutations if and only if the map  $(\varphi  - J_{\psi t}  J_{c})$  is  a permutations of the set $Q$ for any $t\in Q$.

Case 3.  Using  Theorem \ref{ALinear_ORTHOGON_12}  we have: $A\bot A^{23}$ if and only if the maps $\psi^{-1}\varphi - J_t \psi J^{-1}_{\psi^{-1} c}
J_{\psi^{-1} c} I \psi^{-1} \varphi$ are  permutations of the set $Q$.

We simplify the last equality in the following way:
\[
\psi^{-1}\varphi - J_t \psi J^{-1}_{\psi^{-1} c}
J_{\psi^{-1} c} I \psi^{-1} \varphi = \psi^{-1}\varphi - J_t \varphi = \psi^{-1} (\varepsilon - \psi J_t)\varphi.
\]
Therefore $A\bot A^{23}$ if and only if the maps $(\varepsilon + I\psi J_t)$ are permutations of the set $Q$.

Case 4. From Theorem \ref{LEFT_ALINEAR_ORTHOGON} it follows that $A\bot A^{123}$ if and only if the map
\begin{equation}\label{123_ALIN_PAR}
I\varphi^{-1}I\psi + \psi^{-1}\varphi I J_{\varphi^{-1} c} \varphi^{-1}
\end{equation}
is a permutation of the set $Q$. We make the following transformation of  expression (\ref{123_ALIN_PAR}):
\begin{equation}
\begin{split}
& I\varphi^{-1}I\psi + \psi^{-1}\varphi I J_{\varphi^{-1} c} \varphi^{-1} = \\
& \varphi^{-1}\psi - \psi^{-1}  J_{c} \varphi \varphi^{-1} = \\
& \varphi^{-1}\psi - \psi^{-1}  J_{c}.
 \end{split}
\end{equation}
We obtain:  $A\bot A^{123}$ if and only if the map $(\varphi^{-1}\psi - \psi^{-1}  J_{c})$ is a permutation of the set $Q$.

Further we have:
the map $(\varphi^{-1}\psi - \psi^{-1}  J_{c})$ is a permutation of the set $Q$
if and only if the map $\varphi(\varphi^{-1}\psi - \psi^{-1}  J_{c})\psi = (\psi^2 - \varphi  J_{\psi^{-1}c})$ is a permutation of the set $Q$.

Therefore, $A\bot A^{123}$ if and only if the map $(\psi^2 - \varphi  J_{\psi^{-1}c})$ is a permutation of the set $Q$.

Case 5. From Theorem \ref{LEFT_ALINEAR_ORTHOGON} we have: $A\bot A^{132}$  if and only if the map $( \varphi^{-1}\psi - \varphi )$ is a permutation of the set $Q$.

Therefore $A\bot A^{132}$ if and only if the map $( \varphi^{-1}\psi - \varphi )=  \varphi^{-1}(\psi - \varphi^2)$ is a permutation of the set $Q$.
\end{proof}

\begin{corollary}
Any alinear quasigroup over the  group $S_n$ ($n\neq 2; 6$) is not orthogonal to its $(12)$-, $(13)$-, and $(23)$-parastrophe.
\end{corollary}
\begin{proof}
It is possible to use  Theorem \ref{PARASTR_ORTH_OF_ALIN_Q} and Lemma \ref{Lemma_EXAMPLE}.

We  give direct proof. From Case 1 of Theorem \ref{PARASTR_ORTH_OF_ALIN_Q} and properties of the group $S_n$ it follows that there exists an element $w\in S_n$ such that $(\psi^{-1} \varphi - J_w \varphi^{-1} \psi)x = 0$ for any $x\in S_n$.

Cases 2 and 3 are proved in the similar way.
\end{proof}

\begin{theorem} \label{PARASTR_ORTH_OF_LIN_ALIN_Q}
For a left linear  right alinear   quasigroup $(Q,A)$ of the form $A(x, y) =  \varphi x + I\psi y + c$ over a group $(Q, +)$ the
following equivalences are fulfilled:
\begin{enumerate}
\item $A \bot A^{12} \Longleftrightarrow $  the map $(\varphi^{-1} \psi - \psi^{-1} \varphi)$ is a  permutation of the set $Q$;

\item $A\bot A^{13} \Longleftrightarrow $ the map $(\varepsilon + \varphi J_{I c})$ is a  permutation of the set $Q$;

\item $A\bot A^{23} \Longleftrightarrow $ the map  $(J_{t} + \psi)$ is a permutation of the set $Q$ for any $t \in Q$;

\item $A\bot A^{123} \Longleftrightarrow $ the map $(\varphi +  J_{\psi I c}  \psi^2) $ is a permutation of the set $Q$;

\item $A\bot A^{132} \Longleftrightarrow $ the map $(\varphi^2 + I J_k \psi)$ is a  permutation of the set $Q$ for any $k\in Q$.
\end{enumerate}
\end{theorem}
\begin{proof}
The forms of parastrophes of quasigroup $(Q, A)$ are given in Lemma \ref{Forms_of_parastrophes_OF_LIN_ALIN_QUAS}.

Case 1. The proof follows from Theorem \ref{LEFT_LIN_Right_ALINEAR}.

Case 2. Using  Theorem \ref{LEFT_LINEAR_ORTHOGON} we have: $A\bot A^{13}$ if and only if the map $I \varphi^{-1}I\psi + \varphi J_{I\varphi^{_1}c} \varphi^{-1}\psi$ is a  permutation of the set $Q$.

    We make the following transformations: $I \varphi^{-1}I\psi + \varphi J_{I\varphi^{-1}c} \varphi^{-1}\psi =  \varphi^{-1}\psi +  J_{I c} \varphi \varphi^{-1}\psi = (\varphi^{-1} +  J_{I c}) \psi$.
    The last map is a permutation if and only if the map  $(\varphi^{-1} +  J_{I c}) = \varphi^{-1}(\varepsilon + \varphi J_{I c})$  is a permutation of the set $Q$.

Case 3.  Using  Theorem \ref{Linear_LIN_ORTHOGON_12}  we have: $A\bot A^{23}$ if and only if the maps $J_t\varphi^{-1} \psi J^{-1}_{\psi^{-1}c} J_{\psi^{-1}c}\psi^{-1} - \varphi^{-1}I \psi$ are  permutations of the set $Q$.

We simplify the last equality in the following way:
\[
J_t\varphi^{-1} \psi J^{-1}_{\psi^{-1}c} J_{\psi^{-1}c}\psi^{-1} - \varphi^{-1}I \psi = \\
\varphi^{-1}(IJ_{\varphi t } + \psi) = \varphi^{-1}(J_{I\varphi t } + \psi).
\]
Therefore $A\bot A^{23}$ if and only if the maps $(J_{t} + \psi)$ are permutations of the set $Q$ for any $t \in Q$.

Case 4.  From Remark \ref{ONE_MORE_FORM} it follows that $A\bot A^{123}$ if and only if the map
\begin{equation}\label{12354_ALIN_PAR}
\psi^{-1} \varphi + \varphi  J_{I\varphi^{-1} c} \varphi^{-1} \psi = \psi^{-1} \varphi +  J_{I c}  \psi
\end{equation}
is a permutation of the set $Q$.
We obtain:  $A\bot A^{123}$ if and only if the map $(\psi^{-1} \varphi +  J_{I c}  \psi) =\psi^{-1} (\varphi +  J_{\psi I c}  \psi^2) $ is a permutation of the set $Q$, i.e. $A\bot A^{123}$ if and only if the map $(\varphi +  J_{\psi I c}  \psi^2) $ is a permutation of the set $Q$.

Case 5. From Theorem \ref{LEFT_LIN_LEFT_ALINEAR} we have: $A\bot A^{132}$  if and only if the maps $\psi J^{-1}_{\psi^{-1}c} J_{\psi^{-1}c} \psi^{-1} \varphi + J_{I\psi^{-1} b} \varphi^{-1} I \psi =  \varphi + I J_{I\psi^{-1} b} \varphi^{-1} \psi$ are  permutations of the set $Q$ for any $b\in Q$.
Denote the expression $I\psi^{-1} b$ by the letter $t$.

Then $A\bot A^{132}$  if and only if the maps $ \varphi + I J_{t} \varphi^{-1} \psi $ are  permutations of the set $Q$ for any $t\in Q$. But $\varphi + I J_{t} \varphi^{-1} \psi  =  \varphi^{-1}(\varphi^2 + I J_{\varphi t} \psi) $. Therefore $A\bot A^{132}$  if and only if the maps $(\varphi^2 + I J_{\varphi t} \psi) $ are  permutations of the set $Q$ for any $t\in Q$. Denote expression $\varphi t$ by the letter $k$.

Then $A\bot A^{132}$  if and only if the maps $(\varphi^2 + I J_{k} \psi) $ are  permutations of the set $Q$ for any $k\in Q$.
\end{proof}

\begin{corollary}
Any left linear  right alinear   quasigroup over the  group $S_n$ ($n\neq 2; 6$) is not orthogonal to its $(132)$--parastrophe.
\end{corollary}
\begin{proof}
The proof follows from Theorem \ref{PARASTR_ORTH_OF_LIN_ALIN_Q} and Lemma \ref{Lemma_EXAMPLE}. In this case we can find  element $d$
such that $J_d \psi = \varphi^2$.
\end{proof}

\begin{theorem} \label{PARASTR_ORTH_OF_ALIN_RIGHT_LIN_Q}
For a left alinear  right linear   quasigroup $(Q,A)$ of the form $A(x, y) =  I\varphi x + \psi y + c$ over a group $(Q, +)$ the
following equivalences are fulfilled:
\begin{enumerate}
\item $A \bot A^{12} \Longleftrightarrow $  the map $(-\varphi^{-1} \psi + \psi^{-1} \varphi)$ is a  permutation of the set $Q$;

\item $A\bot A^{13} \Longleftrightarrow $ the map $(\varphi + I J_{(Ib+c)})$ is a  permutation of the set $Q$ for any $b\in Q$;

\item $A\bot A^{23} \Longleftrightarrow $ the map $(\psi + \varepsilon)$ is a permutation of the set $Q$;

\item $A\bot A^{123} \Longleftrightarrow $ the map $(\psi^2  + \varphi I J_{t})$ is a permutation of the set $Q$ for any $t\in Q$;

\item $A\bot A^{132} \Longleftrightarrow $ the map  $(\varphi^2 +   \psi)$ is a   permutation of the set $Q$.
\end{enumerate}
\end{theorem}
\begin{proof}
The forms of parastrophes of quasigroup $(Q, A)$ are given in Lemma \ref{Forms_of_parastrophes_OF_ALIN_LIN_QUAS}.

Case 1. The proof follows from Theorem \ref{RIGHT_LIN_LEFT_ALINEAR_U}.

Case 2. Using  Theorem \ref{Linear_LIN_ORTHOGON_012} we have: $A\bot A^{13}$ if and only if the maps $I \psi^{-1}I\varphi  + J_{I\psi^{-1}b} \psi^{-1}  \varphi I J_{\varphi^{-1} c}  \varphi^{-1}$  are  permutations of the set $Q$ for any $b \in Q$.

After simplification we have $I \psi^{-1}I\varphi  + J_{I\psi^{-1}b} \psi^{-1}  \varphi I J_{\varphi^{-1} c}  \varphi^{-1} = \psi^{-1} \varphi + I J_{\psi^{-1}(Ib+c)}\psi^{-1} = \psi^{-1} (\varphi + I J_{(Ib+c)})$.

Case 3.  Using  Theorem \ref{Right_LINEAR_ORTHOGON} we have: $A\bot A^{23}$ if and only if the map $\psi^{-1}I \varphi - \psi\psi^{-1} \varphi =
\psi^{-1}I \varphi + I \varphi = (\varepsilon +\psi^{-1})I \varphi = (\psi + \varepsilon)\psi^{-1}I \varphi$ is a permutation of the set $Q$.

Finally  $A\bot A^{23}$ if and only if the map $(\psi + \varepsilon)$ is a permutation of the set $Q$.

Case 4.  From Theorem \ref{LEFT_LIN_LEFT_ALINEAR}  we have:  $A\bot A^{123}$ if and only if the map
\begin{equation*}
\begin{split}
& \varphi^{-1}\psi + J_{I\psi^{-1} b} \psi^{-1} \varphi I J_{\varphi^{-1}c}\varphi^{-1} = \\
& \varphi^{-1}\psi + I J_{I\psi^{-1} b} J_{\psi^{-1} c} \psi^{-1} \varphi  \varphi^{-1} = \\
& \varphi^{-1}\psi + I J_{\psi^{-1}(I b+c)} \psi^{-1} = \\
& \varphi^{-1} (\psi^2 + \varphi I J_{\psi^{-1}(I b+c)}) \psi^{-1}  \\
\end{split}
\end{equation*}
is a permutation of the set $Q$ for any $b\in Q$. We denote the expression $(\psi^{-1}(I b+c))$ by the letter $t$.
We obtain:  $A\bot A^{123}$ if and only if the map $(\psi^2  + \varphi I J_{t})$ is a permutation of the set $Q$ for any $t\in Q$.

Case 5. From Theorem \ref{RIGHT_LIN_LEFT_ALINEAR_U} we have: $A\bot A^{132}$  if and only if the map $\psi \psi^{-1} \varphi +  \varphi^{-1} \psi =  \varphi +  \varphi^{-1} \psi = \varphi^{-1} (\varphi^2 +   \psi) $ is  a permutation of the set $Q$.
\end{proof}

\begin{remark}
Using results of Lemma \ref{UP_TO_THIRD_COPMONENT} we    can  put $c=0$ in  Theorems \ref{PARASTR_ORTH_OF_ALIN_Q}, \ref{PARASTR_ORTH_OF_LIN_ALIN_Q}, and  \ref{PARASTR_ORTH_OF_ALIN_RIGHT_LIN_Q}.
\end{remark}

\begin{corollary}
Any left alinear  right linear   quasigroup over the  group $S_n$ ($n\neq 2; 6$) is not orthogonal to its $(13)$- and $(123)$-parastrophe.
\end{corollary}
\begin{proof}
The proof follows from Theorem \ref{PARASTR_ORTH_OF_ALIN_RIGHT_LIN_Q} and Lemma \ref{Lemma_EXAMPLE}.

Indeed, we can take into consideration that $A\bot A^{13} \Longleftrightarrow $ the map $(\varphi + I J_{(Ib+c)})$ is a  permutation of the set $Q$ for any $b\in Q$ and the fact that there exists an element $p\in S_n$ such that $(\varphi + I J_{(Ip+c)})x = 0$ for any $x\in Q$.

The second case is proved in the similar way.
\end{proof}

\medskip

\addcontentsline{toc}{section}{\protect{\bf 4 \,\, References}}

\noindent \footnotesize
{Institute of Mathematics and \\
Computer Science \\
Academy of Sciences of Moldova  \\
Academiei str. 5,  MD$-$2028 Chi\c{s}in\u{a}u  \\
Moldova  \\
E-mail: \emph{scerb@math.md }}

\end{document}